\documentclass{amsart}

\usepackage{amssymb,mathtools}
\usepackage[english]{babel}
\usepackage{enumerate}
\usepackage{tikz}
\usepackage[all]{xy}
\usepackage{geometry}
\usepackage{xcolor,colortbl}
\usepackage{array}
\usepackage{multicol,multirow,tabularx}
\usepackage{hyperref}
\hypersetup{
	colorlinks,%
	citecolor=blue,%
	filecolor=blue,%
	linkcolor=blue,%
	urlcolor=blue
}
\usepackage{comment}

\usepackage[utf8]{inputenc}
\usepackage{tikz-cd}
\usepackage[utf8]{inputenc}
\usepackage[T1]{fontenc}
\usepackage{amssymb}
\usepackage[all]{xy}
\usepackage{graphicx}
\usepackage{amsmath}
\usepackage{colonequals}
\usepackage{amsthm}
\usepackage{hyperref}
\usepackage[normalem]{ulem}
\usepackage{comment}
\usepackage{cancel}
\usepackage{ulem} 
\usepackage{url}
\newtheorem{introthm}{Theorem}

\newtheorem{introprop}[introthm]{Proposition}
\newtheorem{prob}[introthm]{Problem}
\newtheorem{thm}{Theorem}[section]
\newtheorem{lem}[thm]{Lemma}
\newtheorem{prop}[thm]{Proposition}
\newtheorem{cor}[thm]{Corollary}

\theoremstyle{definition}

\newtheorem{defi}[thm]{Definition}

\newtheorem{defnot}[thm]{Definition/Notation}
\newtheorem{ex}[thm]{Example}
\newtheorem{introex}[introthm]{Example}

\newtheorem{rem}[thm]{Remark}

\newcommand{\defeq}{\vcentcolon=}
\newcommand{\isom}{\cong}

\newcommand{\rmap}{\dasharrow}
\newcommand{\psmap}{\smash{\xymatrix@C=0.5cm@M=1.5pt{ \ar@{..>}[r]& }}}

\def\dim{\operatorname{dim}}

\def\deg{\operatorname{deg}}

\def\det{\operatorname{det}}

\def\NE{\overline{\operatorname{NE}}}

\def\Exc{\operatorname{Exc}}

\def\Pic{\operatorname{Pic}}

\def\Spec{\operatorname{Spec}}

\def\Bir{\operatorname{Bir}}

\def\Aut{\operatorname{Aut}}

\def\divv{\operatorname{div}}

\def\Amp{\operatorname{Amp}}

\def\disc{\operatorname{disc}}
\def\modd{\operatorname{mod}}

\newcommand{\OO}{\mathcal{O}}
\newcommand{\II}{\mathcal{I}}

\newcommand{\ZZ}{\mathbb{Z}}
\newcommand{\QQ}{\mathbb{Q}}
\newcommand{\RR}{\mathbb{R}}
\newcommand{\CC}{\mathbb{C}}
\newcommand{\PP}{\mathbb{P}}

\definecolor{AGcolor}{RGB}{200,3,3}
\definecolor{olive}{HTML}{3c8031}
\definecolor{myorange}{HTML}{eb811a}
\definecolor{myorchid}{HTML}{9837ba}

\title{On Gizatullin's Problem for quartic surfaces of Picard rank $2$}
\author{Carolina Araujo \and Daniela Paiva \and Sokratis Zikas}
\address{\sc Carolina Araujo\\
	IMPA\\
	Estrada Dona Castorina 110\\
	22460-320 Rio de Janeiro\\ Brazil}
\email{caraujo@impa.br}
\address{\sc Daniela Paiva\\
	IMPA\\
	Estrada Dona Castorina 110\\
	22460-320 Rio de Janeiro\\ Brazil}
\email{da.paiva@impa.br}
\address{\sc Sokratis Zikas\\
	IMPA\\
	Estrada Dona Castorina 110\\
	22460-320 Rio de Janeiro\\ Brazil}
\email{sokratis.zikas@impa.br}

\begin{document}
	
	\maketitle
	
	\begin{abstract}
		In this paper we determine which automorphisms of general smooth quartic surfaces $S\subset \PP^3$ of Picard rank $2$ are restrictions of Cremona transformations of $\PP^3$.
	\end{abstract}
	
	\section*{Introduction}
	In this paper, we address the following question:
	
	\begin{prob}[Gizatullin]\label{G-problem}
		When is a nontrivial automorphism of a smooth quartic surface $S\subset \PP^3$ the restriction of a Cremona transformation of $\PP^3$?
	\end{prob}
	
	\noindent 
	This question has its origin in a classical theorem by Matsumura and Monsky describing automorphisms of smooth hypersurfaces in projective spaces. Let $X_d\subset  \PP^{n+1}$ be a smooth hypersurface of degree $d$. Except for two special cases, namely $(n,d) = (1,3)$ and $(n,d) = (2,4)$, every automorphism of $X_d$ is the restriction of a linear automorphism of the ambient space $\PP^{n+1}$. 
	For $n\geq 2$ this was proved by Matsumura and Monsky \cite{MM64}, and for $n=1$ it is due to Chang  \cite{Chang}. 
	For $(n,d) = (1,3)$, that is, when $C=X_3\subset  \PP^2$ is an elliptic curve, 
	$$
	\Aut(C) \ = \ C \rtimes \ZZ_m, \ \text{ for some } m\in\{2,4,6\},
	$$
	where $C$ acts on itself by translations. 
	Automorphisms of $C$ that are restrictions of linear automorphisms of $\PP^2$ form a finite subgroup of the infinite group $\Aut(C)$. 
	However, one can still describe automorphisms of $C$ in terms of birational self-maps of $\PP^2$. Indeed, 
	every automorphism of $C$ is the restriction of a Cremona transformation of  $\PP^2$  \cite[Theorem 2.2]{Oguiso2}. 
	One may wonder if the same is true for the other exceptional case $(n,d) = (2,4)$, that is, when $S=X_4\subset  \PP^3$ is a smooth quartic surface. This is the context of Problem~\ref{G-problem}.

	For a smooth quartic surface $S\subset  \PP^3$ with Picard rank $\rho(S)=1$, Problem~\ref{G-problem} is not very interesting, as in this case $\Aut(S) = \{id\}$ \cite[Chapter 15, Corollary 2.12]{Huybrechts}. 
	For quartic surfaces with higher Picard rank, the problem was first addressed by Oguiso in \cite{Oguiso2} and \cite{Oguiso1}.
	In \cite{Oguiso2}, Oguiso provided a positive example:
	he constructed a smooth quartic surface $S\subset \PP^3$ with $\rho(S)=3$ and $\Aut(S)\cong\ZZ_2*\ZZ_2*\ZZ_2$ such that every automorphism of $S$ is the restriction of a Cremona transformation of $\PP^3$. 
	The involutions generating $\Aut(S)$ in this example are shown to preserve special elliptic fibrations on $S$, a property that is exploited to show that they are restrictions of Cremona transformations of $\PP^3$.
	On the negative side, Oguiso constructed in \cite{Oguiso1} a smooth quartic surface $S\subset \PP^3$ with Picard rank $\rho(S)=2$, $\Aut(S)\cong\ZZ$, and such that no nontrivial automorphism of $S$ is induced by a Cremona transformation of $\PP^3$.
	A key idea utilized in this example is the following consequence of the \emph{Sarkisov program} due to Takahashi \cite{Takahashi}:
	the existence of a Cremona transformation stabilizing a quartic surface $S\subset \PP^3$ forces the existence of a curve $C\subset S$ of degree $d<16$ that is not a complete intersection of $S$ with another surface in $\PP^3$, a property not satisfied in Oguiso's example.
	This approach was pushed further by Paiva and Quedo in \cite{paivaquedo} to investigate Gizatullin’s Problem for smooth quartic surfaces $S\subset \PP^3$ with $\rho(S)=2$ more generally. 
	In this case, for a suitable basis of $\Pic(S)\cong \ZZ^2$, the intersection product in $S$ is given by a matrix of the form 
	\[
	\left(\begin{array}{cc}
		4 & b \\
		b & 2c
	\end{array}\right),
	\]
	with $b,c\in \ZZ$. We denote by $r=b^2-8c$ the discriminant of $S$. 
	By \cite{paivaquedo}, if $r> 233$, then there is no curve $C\subset S$ of degree $d<16$ that is not a complete intersection of $S$ with another surface in $\PP^3$. Hence, in those cases, the group $\Bir(\PP^3;S)$ of Cremona transformation of $\PP^3$ stabilizing $S$ coincides with the group $\Aut(\PP^3;S)$ of automorphisms of $\PP^3$ stabilizing $S$.

	The aim of this paper is to complete this analysis and solve Problem~\ref{G-problem} for quartic surfaces $S$ with Picard rank $2$. Our main result is the following. We refer to Section~\ref{sec:aut_general} for the definition of the finite index subgroup $\Aut^{\pm}(S)$ of $\Aut(S)$. 
	Here we only mention that a very general K3 surface satisfies $\Aut(S)= \Aut^{\pm}(S)$ (see Remark \ref{rem:generalIsAutGeneral}). 
	In this case, the surface is called \emph{Aut-general}.
	
	\begin{introthm}\label{introthm}
		Let $S\subset \PP^3$ be a smooth quartic surface over $\CC$ with Picard rank $2$ and discriminant $r$. 
		\begin{enumerate}
			\item If $r > 57$ or $r =52$, then $\Bir(\PP^3;S)=\Aut(\PP^3;S) =\{id\}$. 
			\item If $r \leq 57$ with $r\neq 52$, then every element of the subgroup $\Aut^{\pm}(S) \subseteq \Aut(S)$ is the restriction of a nonregular Cremona transformation of $\PP^3$. 
			Moreover, one of the following holds:
			\begin{itemize}
				\item $r \in  \{9, 12, 16, 24, 25, 33, 36, 44, 49, 57\}$ and $\Aut(S) = \Aut^{\pm}(S) = \{id\}$;
				\item $r \in  \{17, 41\}$ and $\Aut(S) = \Aut^{\pm}(S) \cong \ZZ_2$;
				\item $r \in  \{28, 56\}$ and $\Aut^{\pm}(S) \cong \ZZ_2 \ast \ZZ_2$; or
				\item $r \in  \{20, 32, 40, 48\}$ and $\Aut^{\pm}(S) \cong \ZZ$.
			\end{itemize}       
		\end{enumerate}
	\end{introthm}

	Let us briefly explain the strategy of the proof of Theorem~\ref{introthm}. 
	Our main tool is a special version of the Sarkisov program for Calabi-Yau pairs.
	The classical Sarkisov program allows one to factorize any birational self-map of $\PP^n$ as a composition of simple birational maps  between Mori fiber spaces, called  \emph{Sarkisov links} (see Section~\ref{section:birational} for details).  
	The special version for Calabi-Yau pairs allows one to factorize any \emph{volume preserving} birational self-map of a Calabi-Yau pair of the form $(\PP^n, D)$ as a composition of volume preserving Sarkisov links. We refer to Subsection~\ref{Subsection:CYpairs} for the definition of Calabi-Yau pair and volume preserving birational maps. 
	Here, we only mention that when $S \subset \PP^3$ is a smooth quartic surface, the pair $(\PP^3, S)$ is a \emph{canonical} Calabi-Yau pair. In this special case, a Cremona transformation of $\PP^3$ is volume preserving for the pair $(\PP^3, S)$ if and only if it restricts to an automorphism of $S$ (see Remark~\ref{rem:decomposition_group}). This feature makes this special version of the Sarkisov program particularly convenient in the context of Gizatullin’s problem.
	The volume preserving condition imposes strong restrictions.
	In particular, it forces the first link in the decomposition to start with the blowup of a curve $C\subset S$ (Proposition~\ref{CinS}). 
	While curves in $\PP^3$ whose blowups initiate Sarkisov links are not completely classified yet, a key technical result that we prove is the following classification of space curves \emph{contained in smooth quartic surfaces with Picard rank $2$} whose blowups initiate Sarkisov links.

	\begin{introprop}[=Proposition~\ref{prop:blancLamyCurves}] \label{intro_prop}
		Let $C\subset \PP^3$ be a (possibly singular) curve of arithmetic genus $p_a$ and degree $d$ lying on a smooth quartic surface with Picard rank $2$,
		and denote by $X \to \PP^3$ the blowup of $\PP^3$ along $C$.
		If $X \to \PP^3$ initiates a Sarkisov link, then $X$ is weak Fano (i.e., $-K_X$ is nef and big) and 
		\[
		(p_a,d) \in 
		\left\{
		\def\arraystretch{1.2}
		\begin{array}{c}
			\arraycolsep=1.4pt\def\arraystretch{1.2}
			\begin{array}{llllllllllll}
				(0,1), & (0,2), & (0,3), & (0,4), & (0,5), & (0,6), & (0,7), & (1,3), & (1,4), & (1,5), & (1,6), & (1,7),\\
				(2,5), & (2,6), & (2,7), & (2,8), & (3,6), & (3,7), &
				(3,8), & (4,6), & (4,7), & (4,8), & (5,7), & (5,8), 
			\end{array}\\
			\arraycolsep=2.6pt
			\begin{array}{llllllllll}
				{(6,8)}, & (6,9), & (7,8), & (7,9), & (8,9), & (9,9), & (10,9), & (10,10), & (11,10), & (14,11)
			\end{array}
		\end{array}
		\right\}.
		\]
	\end{introprop}

	Let $S$ be a smooth quartic surface with Picard rank $2$ and discriminant $r$, and suppose that $\Bir(\PP^3;S)\setminus \Aut(\PP^3;S)\neq \emptyset$.
	Proposition~\ref{intro_prop} above implies that $S$ contains a curve $C$ with arithmetic genus and degree $(p_a,d)$ in the above list. 
	The discriminant of the sublattice of $\Pic(S)$ generated by the hyperplane class and the curve $C$ can be readily computed and must be divisible by $r$.
	This implies that $r \leq 57$ and $r\neq 52$ (Corollary~\ref{cor:noRealizationForr>57}).
	Combined with Corollary~\ref{cor:actionDetermByPic}(3), this gives the first part of Theorem~\ref{introthm}. 
	
	In order to realize the automorphisms of $S\subset \PP^3$ as Cremona transformations and prove the second part of Theorem~\ref{introthm}, we introduce a new approach, again exploiting Sarkisov links. 
	Let us explain the general strategy.
	First of all, for each discriminant $r \leq 57$, $r\neq 52$, we determine $\Aut^{\pm}(S)$ for a smooth quartic surface $S$ with Picard rank $2$ and discriminant $r$. 
	This can be done by verifying the existence of certain numerical classes in $\Pic(S)$, using results of \cite{GLP} and \cite{paivaquedo} (Proposition~\ref{prop:Autforrleq57}). The possibilities are the following:
	\[
	\Aut^{\pm}(S) \cong 
	\left\{
	\begin{array}{ll}
		\{id\},            & \text{ if } r \in \{9, 12, 16, 24, 25, 33, 36, 44, 49, 57\};\\
		\ZZ_2,            & \text{ if } r \in \{17, 41\};\\
		\ZZ_2 \ast \ZZ_2, & \text{ if } r \in \{28, 56\}; \\
		\ZZ,              & \text{ if } r \in \{20, 32, 40, 48\}.    
	\end{array}   
	\right.
	\]
	After describing the action of the generators of $\Aut^{\pm}(S)$ on $\Pic(S)$, we proceed to construct Cremona transformations realizing them.
	We do so by considering the Sarkisov links initiated by blowing up curves $C\subset S$ with invariants $(p_a,d)$ listed in Proposition~\ref{intro_prop}. These links have been well studied, and many of them are already Cremona transformations $\PP^3\dasharrow \PP^3$.
	
	\begin{introex}[$r=41$]\label{ex:r=41}
		Let $S\subset \PP^3$ be a smooth quartic surface with Picard rank $2$ and discriminant $r=41$.
		Denote by $H\in \Pic(S)$ the hyperplane class, and by $\iota=\varphi_{|H|}\colon S\to \PP^3$ the corresponding embedding. 
		We show in Proposition~\ref{prop:Autforrleq57} that $\Aut(S) \cong \ZZ_2$.
		There are exactly 2 classes in $\Pic(S)$ corresponding to curves in $S$ with arithmetic genus and degree $(p_a,d)$ listed in Proposition~\ref{intro_prop}, namely $(p_a,d)=(2,7)$ and $(p_a,d)=(6,9)$.
		Indeed, in Proposition~\ref{prop:boldcurvesasgenerator} we prove the existence of a curve $C\subset S$ with $(p_a,d)=(6,9)$. The case  
		$(p_a,d)=(2,7)$ can be treated in the same way. 
		
		Consider a smooth curve $C\subset S$ with $(p_a,d)=(6,9)$. 
		We show in Proposition~\ref{prop:boldcurvesasgenerator} that $\{H,C\}$ is a basis of $\Pic(S)$, and that 
		the action of the generator $g$ of $\Aut(S)$ on $\Pic(S)$ with respect to this basis is given by the following matrix:
		\[
		\left(\begin{array}{rr}
			27 & 104\\
			-7 & -27
		\end{array}\right).
		\]  
		The Sarkisov link initiated by the blowup $p\colon X\to \PP^3$ of the curve $C$ was described in detail in \cite[Proposition 3.1 and Remark 3.2]{ZikRigid}. It is 
		a birational involution $\varphi\colon \PP^3\dasharrow \PP^3$ that preserves the quartic surface $S$ (after composition with a suitable automorphism of $\PP^3$). More precisely, $\varphi$ fits into a commutative diagram:
		\[
		\xymatrix@R=.3cm{
			X \ar[dd]_p \ar[rd] \ar@{..>}^{\tilde \varphi}[rr] && X \ar[ld] \ar[dd]^p\\
			& Z \ar@(dl,dr)^{\alpha}\\
			\PP^3 \ar@{-->}_{\varphi}[rr] && \ \PP^3 \ ,
		}
		\]
		where $\tilde \varphi \colon X\psmap X$ is a flop, the anti-canonical model $Z$ of $X$ is a double cover of $\PP^3$ ramified along a sextic hypersurface, and $\alpha \colon Z\to Z$ is the deck transformation of $Z$ over $\PP^3$.
		In order to check that the birational involution $\varphi$ of $\PP^3$ restricts to the involution $g$ of $S$, we verify that the linear systems on $S$ corresponding to the embeddings $\varphi\circ \iota \colon S\to \PP^3$ and $\iota\circ g \colon S\to \PP^3$ coincide:
		\[
		\xymatrix@R=.5cm@C=1.3cm{
			\PP^3 \ar@{-->}[r]^{\varphi} & \PP^3\\
			S \ar[u]^{\iota} \ar[r]_{g} & \, S \, . \ar[u]_{\iota}
		}
		\]
		
		For the sake of completeness, we mention that the Sarkisov link initiated by the blowup of a smooth curve with invariants $(p_a,d)=(2,7)$ was described in \cite[Example 5.12]{BL}. It is a birational map $\PP^3\dasharrow V_4 \subset \PP^5$, where $V_4$ denotes a complete intersection of two quadrics in $\PP^5$.
	\end{introex}

	When $S\subset \PP^3$ is a smooth quartic surface with Picard rank $2$ and discriminant $r \in \{17, 28, 56\}$, $\Aut^{\pm}(S)$ is also generated by involutions, and we follow the same strategy described in Example~\ref{ex:r=41} to realize these generators as birational involutions of $\PP^3$. 
	On the other hand, when $r \in \{20, 32, 40, 48\}$, one has $\Aut^{\pm}(S)\cong \ZZ$. 
	In these cases, the construction of a Cremona transformation of $\PP^3$ restricting to a generator of $\Aut^{\pm}(S)$ also uses Sarkisov links, but it is more involved. We explain it in a special case. 
	
	\begin{introex}[$r=48$]\label{ex:r=48}
		Let $S\subset \PP^3$ be a smooth quartic surface with Picard rank $2$ and discriminant $r=48$.
		Denote by $H\in \Pic(S)$ the hyperplane class. By Proposition~\ref{prop:Autforrleq57}, $\Aut^{\pm}(S) \cong \ZZ$, and we want to realize the generator $g$ of $\Aut^{\pm}(S)$ as the restriction of a Cremona transformation of $\PP^3$. 
		There are exactly 2 classes in $\Pic(S)$ corresponding to curves $C$ and $C'$ in $S$ with arithmetic genus and degree $(p_a,d)=(3,8)$, and this is the only pair $(p_a,d)$ listed in Proposition~\ref{intro_prop} realized by curves in $S$ (Proposition~\ref{prop:boldcurvesasgenerator}). Moreover, the curves $C$ and $C'$ satisfy $C+C'=4H$ in $\Pic(S)$.
		The action of the generator $g$ of $\Aut^{\pm}(S)$ on $\Pic(S)$, with respect to the basis $\{H,C\}$, is given by the following matrix (see Proposition~\ref{prop:boldcurvesasgenerator}):
		\[
		\left(\begin{array}{rr}
			209 & 56\\
			-56 & -15
		\end{array}\right).
		\]  
		The Sarkisov link $\varphi_1\colon \PP^3\dasharrow \PP^3$ initiated by the blowup $p\colon X\to \PP^3$ of a smooth curve $C$ of genus and degree $(g,d)=(3,8)$ is described in \cite[Table 2]{BL}. It fits into a commutative diagram:
		\[
		\xymatrix@R=.5cm@C=1.3cm{
			X \ar@{..>}[r]^{\chi} \ar[d]_p  &  \, X_1\ar[d]^{p_1} \\
			\PP^3  \ar@{-->}[r]^{\varphi_1} &  \, \PP^3  \, ,
		}
		\]
		where $\chi$ is a flop and $p_1\colon X_1 \to \PP^3$ is the blowup of $\PP^3$ along another smooth curve $C_1$ of genus and degree 
		$(g,d)=(3,8)$.
		One can check that $\varphi_1$ maps $S$ isomorphically onto its image $S_1\subset \PP^3$ and $C_1\subset S_1$ (see Proposition~\ref{prop:realizationZ}).
		However, $S$ and $S_1$ are not projectively equivalent, and so  $\varphi_1$ is not the Cremona transformation that we are looking for. 
		We then choose a smooth curve $C_2\subset S_1$ in the linear system $|4H_1-C_1|$, where $H_1\in \Pic(S_1)$ denotes the hyperplane class. The curve $C_2$ also has  genus and degree $(g,d)=(3,8)$.
		The Sarkisov link $\varphi_2\colon \PP^3 \dasharrow \PP^3$ initiated by the blowup $p_2\colon X_2\to \PP^3$ of the curve $C_2$ is described exactly as before.
		We consider the composition $\varphi=\varphi_2\circ \varphi_1$:
		\[
		\xymatrix@R=.5cm{
			X \ar@{..>}[r] \ar[d]_{p}  &  \, X_1\ar[dr]^{p_1} && X_2\ar[dl]_{p_2} \ar@{..>}[r] &  X'  \ar[d]^{p'}  \\
			\PP^3  \ar@{-->}[rr]^{\varphi_1}\ar@{-->}@/_2.0pc/[rrrr]^{\varphi} &&   \PP^3 \ar@{-->}[rr]^{\varphi_2}  &&  \, \PP^3  \, .
		}
		\]
		\smallskip
		
		\noindent We denote by $S'\subset \PP^3$ the image of $S$ under $\varphi$.
		We can check that $S$ and $S'$ are projectively equivalent. So, after composition with a suitable automorphism of $\PP^3$, we have $\varphi\in \Bir(\PP^3;S)$. Moreover, as in the previous example, we verify that $\varphi$ restricts to the generator $g$ of $\Aut^{\pm}(S)$ by comparing the linear systems that give $\varphi_{|S}$ and $g$. 
	\end{introex}

	In conclusion, our results suggest that the answer to Problem~\ref{G-problem} for smooth quartic surfaces with Picard rank at least two should be: \emph{almost never}. One may even hope for a complete classification of automorphisms of smooth quartic surfaces that are restrictions of Cremona transformations of $\PP^3$. 
	On the other hand, as far as we know, there is currently no classification of the possible automorphism groups of K3 surfaces with higher Picard rank, even when the surface is Aut-general. 
	It should be also noted that in all the cases for which we provide a positive answer to Problem~\ref{G-problem}, we only describe one way to realize each automorphism in $\Aut^{\pm}(S)$.
	More specifically, we only exploit the Sarkisov links initiated by the blowups of \emph{some} classes of curves with invariants $(p_a,d)$ listed in Proposition~\ref{intro_prop}.
	We expect that many of these automorphisms admit other Sarkisov decompositions, which raises the interesting problem of describing the structure of the group $\Bir(\PP^3;S)$, as well as the kernel of the restriction homomorphism $\Bir(\PP^3;S) \to \Aut(S)$.
	These groups are classically known as the \emph{decomposition} and  \emph{inertia} groups of the quartic surface $S\subset \PP^3$,
	respectively. 
	
	\medskip

	This paper is organized as follows:
	in Section~\ref{section:K3s}, we review some of the basic theory of K3 surfaces and their automorphisms, focusing on smooth quartic surfaces of Picard rank $2$;
	in Section~\ref{section:birational}, we review the Sarkisov Program, both its classical version and the volume preserving version for Calabi-Yau pairs;
	in particular, we prove Proposition~\ref{intro_prop} classifying space curves contained in smooth quartic surfaces with Picard rank $2$ whose blowups initiate Sarkisov links; 
	finally, we put these tools together in Section~\ref{section:main} to prove Theorem~\ref{introthm}.

	\medskip
	
	\paragraph{\textbf{Notation \& Conventions}}
	In this paper we always work over $\CC$. 
	All varieties are assumed to be projective and irreducible.
	In particular, all K3 surfaces are assumed to be projective.
	By a curve we always mean an irreducible and reduced curve.
	We often abuse notation and use the same symbol to denote a curve $C$ on a surface $S$ and the corresponding class in $\Pic(S)$.
	While we write $X\dasharrow Y$ to denote an arbitrary birational map between varieties $X$ and $Y$, we write $X\psmap Y$ to denote a small birational map between $X$ and $Y$, i.e.,\ a birational map that is an isomorphism in codimension one. 
	
	\medskip

	\paragraph{\textbf{Acknowledgements}} 
	Gizatullin’s problem was the topic of the Working Group in Algebraic Geometry during the BIRS' ``Latin American and Caribbean Workshop on Mathematics and Gender'' at Casa Matem\'atica Oaxaca in 2022. We thank the participants of this working group for many rich discussions, especially Michela Artebani, Paola Comparin, Alice Garbagnati, Ana Quedo and Cec\'\i lia Salgado. 
	We also thank Fabio Bernasconi, Enrica Floris, Alex Massarenti, Alessandra Sarti and Nikolaos Tsakanikas.
	Finally, we are grateful to the referee for the valuable suggestions that helped improve the presentation of the paper.
	Carolina Araujo has been supported by grants from CNPq, Faperj and CAPES/COFECUB. Daniela Paiva  has been supported by CAPES.
	Sokratis Zikas acknowledges the support of the Swiss National Science Foundation Grant ``Cremona groups of higher rank via the Sarkisov program'' P500PT\_211018 and would like to thank the Instituto de Matem{\'a}tica Pura e Aplicada for their hospitality.

	\section{K3 surfaces and their automorphisms}\label{section:K3s}

	\subsection{The lattices of a K3 surface and the global Torelli Theorem}
	
	\begin{defi}\label{def:lattice}
		A lattice $L$ is a free $\ZZ$-module of finite rank equipped with a non-degenerate symmetric bilinear form $q$.
		We denote by the same symbol the extensions of $q$ to a bilinear form on $L\otimes_{\ZZ} \QQ$ or $L\otimes_{\ZZ} \RR$.
		\begin{enumerate}
			\item The \emph{discriminant} of $L$ is
			\[
			\disc(L) = -\det(Q),
			\]
			where $Q$ is the matrix of $q$ with respect to any basis of $L$.
			\item The lattice $L$ is said to be \emph{even} if $q(x,x)\in 2\ZZ$ for any $x\in L$.
			\item The \emph{signature} of $L$ is the signature of the extended bilinear form $q$ on $L\otimes_{\ZZ} \RR$. 
			\item The \emph{orthogonal group}  of $L$ is  the group $O(L)$ of isometries of $L$, i.e.,  isomorphisms of $L$ preserving the bilinear form $q$.
			\item The \emph{dual lattice} of $L$ is
			\[
			L^* = \big\{ x \in L \otimes_{\ZZ} \QQ \, \big| \, \forall y \in L \, ,\  q(x,y) \in \ZZ\big\}\supset L.
			\]
			
			\item The \emph{discriminant group}  of $L$ is the quotient $A(L)=L^*/L$.
			\item An isometry $\varphi$ of $L$ induces an automorphism of the discriminant group $A(L)$, which we denote by $\overline{\varphi}$.
			
		\end{enumerate}
	\end{defi}

	\begin{defi}
		A \emph{K3 surface} is a smooth projective surface $S$ such that $H^1(S,\mathcal{O}_S)=0$ and $H^0(S,\Omega_S^2) = \CC \cdot \omega_S$ for a nowhere vanishing $2$-form $\omega_S$.
		
	\end{defi}
	
	Let $S$ be a K3 surface. There are two natural lattices associated to $S$: the Picard group $\Pic(S)$, endowed with the intersection product, and the second cohomology group $H^2(S,\ZZ)$, equipped with the cup product. The exponential sequence identifies the Picard group $\Pic(S)\cong H^1(S,\OO^*_S)$ with a sublattice of $H^2(S,\ZZ)$, described as $\Pic(S)=\big\{x\in H^2(S,\ZZ)\ | \ \langle x,\omega_S\rangle=0\big\}$, where the restriction of the cup product coincides with the intersection product of $S$.
	The Picard group $\Pic(S)$ is an even lattice of signature $(1,\rho(S)-1)$, 
	where $\rho(S)$ is the Picard rank of $S$. Conversely, we have the following.
	
	\begin{thm}[{\cite[Corollary 2.9]{Morrison}}]\label{thm:K3andlattices}
		Let $L$ be an even lattice of signature $(1,\rho-1)$, with $\rho\leq 10$. Then there exists a K3 surface $S$ and a lattice isometry $\Pic(S)\cong L$.
	\end{thm}
	
	The second cohomology group $H^2(S,\ZZ)$ has a weight-two Hodge structure given by the decomposition
	$$H^2(S,\mathbb{C})=H^{2,0}(S)\oplus H^{1,1}(S)\oplus H^{0,2}(S).$$
	Here $H^{2,0}(S)\cong H^0(S,\Omega_S^2)$ and, after scalar extension of the bilinear form,  $H^{2,0}(S)\oplus H^{0,2}(S)$ is orthogonal to $H^{1,1}(S)$. 
	
	\begin{defi}
		Let $S$ be a K3 surface.
		\begin{enumerate}
			\item The \emph{discriminant} of $S$ is the discriminant of its Picard lattice $\Pic(S)$.
			\item The \emph{ample cone} $\Amp(S)\subset \Pic(S)\otimes \RR$ is the set of all finite sums $\sum a_iD_i$ with $D_i\in \Pic(S)$ ample and $a_i\in\RR_{> 0}$. 
			\item The \emph{transcendental lattice of $S$} is the orthogonal complement $T(S)\defeq\Pic(S)^{\perp}\subset H^2(S,\ZZ)$.
			\item A \emph{Hodge isometry} of $H^2(S,\ZZ)$ is an isometry $\varphi\in O(H^2(S,\ZZ))$ whose extension to $H^2(S,\mathbb{C})$ preserves the Hodge decomposition.
		\end{enumerate}
	\end{defi}
	
	For a K3 surface $S$, the discriminant groups of $\Pic(S)$ and $T(S)$ are naturally isomorphic, $A(\Pic(S))\cong A(T(S))$.
	Any Hodge isometry $\varphi$ of $H^2(S,\ZZ)$ preserves the lattices $\Pic(S)$ and $T(S)$, inducing isometries $\varphi_1$ and $\varphi_2$ on $\Pic(S)$ and $T(S)$, respectively, such that $\overline{\varphi_1}=\overline{\varphi_2}$ under the identification $A(\Pic(S))\cong A(T(S))$.     
	The following result is a converse of this last fact. We refer to \cite[Corollary 1.5.2]{Nikulin2} for a more general result. 
	
	\begin{thm}[Gluing isometries]\label{thm:gluing} Let $S$ be a K3 surface.
		Let $\varphi_1$ and $\varphi_2$ be isometries of $\Pic(S)$ and $T(S)$, respectively. If $\overline{\varphi_1}=\overline{\varphi_2}$ under the identification $A(\Pic(S))\cong A(T(S))$, then there exists an isometry $\varphi$ on $H^2(S,\ZZ)$ whose restrictions to $\Pic(S)$ and $T(S)$ are $\varphi_1$ and $\varphi_2$, respectively.    
	\end{thm}
	
	We end this subsection with the global Torelli Theorem, which allows us to study  automorphisms  of $S$ via isometries of $H^2(S,\ZZ)$. A reference for it is {\cite[Chapter 7, Theorem 5.3]{Huybrechts}}.
	
	\begin{thm}[Global Torelli Theorem]\label{thm:torelli} Let $S$ be a K3 surface.
		Let $\varphi:H^{2}(S,\mathbb{Z})\longrightarrow H^2(S,\mathbb{Z})$ be a Hodge isometry sending an ample class to an ample class. Then there exists a unique automorphism $f$ of $S$ such that $f^*=\varphi$.    
	\end{thm}

	\subsection{Curves and Linear Systems on quartic surfaces}
	
	In this subsection we gather some results about curves and linear systems on K3 surfaces that we will use in the following sections. 
	
	\begin{thm}[{\cite[Theorem 1]{MoriK3}, \cite[Theorem 1.1]{Knutsen}}]
		\label{thm:gdBounds}
		Let $(g,d)$ be a pair of integers, with $g\geq 0$ and $d>0$.
		Then there exists a smooth curve $C$ of genus $g$ and degree $d$ contained in some smooth quartic surface $S$ if and only if $g= \frac{d^2}{8} + 1$, or $g < \frac{d^2}{8}$ and $(g,d)\neq (3,5)$.
		Moreover, if $(g,d)$ is not the degree and genus of the complete intersection of $S$ with another surface in $\PP^3$, then $g < \frac{d^2}{8}$, and $C$ and $S$ can be chosen so that $\Pic(S) = \ZZ H \oplus \ZZ C$, where $H$ is the class of a hyperplane section.
	\end{thm}

	\begin{prop}
		[{\cite[Theorem 5]{MoriK3}}]\label{prop:criterioveryample} Let $D$ be a nef divisor on a K3 surface $S$ such that $D^2\geq4$. Then $D$ is very ample if and only if the following three conditions hold:
		\begin{enumerate}
			\item There is no irreducible curve $E\subset S$ such that $E^{2}=0$ and
			$D \cdot E\in\{1,2\}$.
			\item There is no irreducible curve $E\subset S$ such that $E^{2}=2$, and $D \sim 2E$.
			\item There is no irreducible curve $E\subset S$ such that $E^{2}=-2$ and
			$E \cdot D=0$.
		\end{enumerate}
	\end{prop}
	\begin{lem}\label{lem:noFixedPart}
		Let $D \neq 0$ be an effective divisor on a K3 surface $S$ such that $D^2 \geq 0$, and let $H$ be an ample class.
		Write $D = M + F$, where $M$ and $F$ are the mobile and fixed parts of $D$, respectively.
		If $M$ and $H$ are not proportional and $F \neq 0$, then $M \not\in \ZZ H \oplus \ZZ D \subset \Pic(S)$.
		In particular, if $S$ is a smooth quartic surface with $\Pic(S) = \ZZ H \oplus \ZZ D$, where $H$ is the class of the hyperplane section and $D$ is a divisor with $D^2\geq 0$, then $D$ has no fixed component,
		unless $\disc(S)=9$ and $D$ is not nef.
	\end{lem}

	\begin{proof}
		Set $d = H \cdot D$, $d_M = H\cdot M$ and $\alpha = H^2$. 
		Suppose that $M$ and $H$ are not proportional, and that $M$ lies in the lattice spanned by $H$ and $D$.   
		Then $H$ and $M$ span a sublattice of $\ZZ H \oplus \ZZ D$ of rank $2$ and so, 
		\[
		{d_M}^2 - \alpha M^2 = \disc \left( \ZZ H \oplus \ZZ M \right) = k\disc \left( \ZZ H \oplus \ZZ D \right) = k(d^2 - \alpha D^2)
		\]
		for some integer $k$.
		By the Hodge index theorem, we have $d^2 - \alpha D^2>0$ and  $k \geq 1$.
		Since $D = M + F$ and $H$ is ample, we have $d_M \leq d$.
		Moreover, $h^0(S,D) = h^0(S,M)$ and, since $D\neq 0$, we have
		\[
		h^0(S,M) = h^0(S,D) \geq \chi(S,D) = \frac{D^2}{2} + 2.
		\]
		Since $M$ is effective with no fixed component, it is nef and big, and so, by the Kawamata-Viewheg vanishing theorem we get
		\[ 
		h^0(S,M) = \frac{M^2}{2} + 2.
		\]
		Combining the inequalities we get
		\[
		k(d^2 - \alpha D^2) = d_M^2 - \alpha M^2 \leq d^2 - \alpha D^2,
		\]
		which can only be satisfied when $k = 1$ and, more importantly, when $d = d_M$, i.e.,  when $D$ has no fixed part.

		For the second part, write $D = M + F$, where $M$ and $F$ are the mobile and fixed parts of $D$, respectively.
		If $M$ and $H$ are not proportional, we conclude that $F=0$ by the first part.
		So we treat the case $M = \lambda H$, $\lambda\geq 1$.
		Suppose that $F \neq 0$. 
		It follows from the discussion above that 
		\[
		4\lambda^2+F^2+2\lambda H\cdot F = (\lambda H + F)^2 = D^2\leq M^2=(\lambda H)^2=4\lambda^2,
		\]
		and thus $F^2\leq - 2\lambda H\cdot F <0$. 
		Since $\{H,D\}$ span $\Pic(S)$, so does $\{H,F\}$.
		Let $\Gamma$ be a curve in the support of $F$.
		Developing the equality
		\[
		\disc\left(\ZZ H \oplus \ZZ \Gamma \right) = k \disc\left(\ZZ H \oplus \ZZ F \right)
		\]
		as above, we get that $k = 1$ and $H\cdot F = H\cdot \Gamma$, and so $F = \Gamma$, i.e.,  $F$ is irreducible and reduced.
		Moreover,
		\[
		-2=F^2\leq - 2\lambda H\cdot F  \implies \lambda = H\cdot F = 1
		\] 
		Computing the discriminant of $S$ with respect to $\{H,F\}$, we get that $\disc(S) = 9$.
		In that case, $D\cdot F = (H + F)\cdot F = -1$.
	\end{proof}
	
	Combining Lemma \ref{lem:noFixedPart} with \cite[Corollary 3.2]{SD74}, which says that a linear system on a K3 surface has no base points outside its fixed components, we get the following:
	
	\begin{cor}\label{cor:linearSystem}
		Let $S$ be a smooth quartic surface with $\Pic(S) = \ZZ H \oplus \ZZ D$, where $H$ is the class of a hyperplane section and $D$ is an effective nef divisor with $D^2 \geq 0$. 
		Then $|D|$ is base point free, and thus a general element in $|D|$ is a smooth curve.
	\end{cor}
	
	\subsection{Aut-general K3 surfaces}\label{sec:aut_general}
	Let $S$ be a K3 surface, and write $H^0(S,\Omega_S^2) = \CC \cdot \omega_S$.
	It follows from Theorem \ref{thm:gluing} and Theorem \ref{thm:torelli} that any automorphism $g$ of $S$ is completely determined by its action on $\Pic(S)$ and its action on $\omega_S$ (see for instance \cite[Chapter 3, Lemma 3.3]{Huybrechts}).
	The action of $g$ on $\omega_S$ is always of the form $g^*\omega_S = \zeta\omega_S$ for some root of unity $\zeta$. 
	When $\zeta=1$, we say that $g$ is \emph{symplectic}, and when $\zeta=-1$, we say that $g$ is \emph{anti-symplectic}.
	We now introduce a distinguished finite index subgroup of $\Aut(S)$, and a generality condition for K3 surfaces that is particularly useful when studying their automorphisms. 
	
	\begin{defi}\label{def:Autgeneral}
		We introduce the subgroup of $\Aut(S)$ consisting of symplectic and anti-symplectic automorphisms:
		\[
		\Aut^{\pm}(S)\ := \ \big\{g \in \Aut(S) \ \big| \ g^*\omega_S = \pm\omega_S \big\}.
		\]
		By \cite[Theorem 10.1.2 a)]{Nikulin3}, $\Aut^{\pm}(S)$ is a finite index subgroup of $\Aut(S)$. 
		
		We say that $S$ is \emph{Aut-general} if $\Aut^{\pm}(S)=\Aut(S)$.
	\end{defi}
	
	\begin{rem}\label{rem:generalIsAutGeneral}
		As the name suggests, being Aut-general is a generality condition on the period domain of K3 surfaces with a fixed embedding of a lattice $L \hookrightarrow \Pic(S)$ (see \cite[Theorem 10.1.2 c)]{Nikulin3}). 
	\end{rem}
	
	Any K3 surface with odd Picard rank is Aut-general (see for instance \cite[Chapter 3, Corollary 3.5]{Huybrechts}).	Explicit examples of non-Aut-general K3 surfaces can be found, for instance, in \cite{ArtebaniSarti}. 
	We mention here the following two particularly interesting cases.
	The first is a smooth quartic surface $S\subset \PP^3$ with $\rho(S)=4$.
	The second is a K3 surface $S'$ with $\rho(S')=2$, which cannot be embedded as a quartic surface in $\PP^3$.
	We do not know of any example of a smooth quartic surface in $\PP^3$ of Picard rank $2$ that is not Aut-general, 
	however Corollary \ref{cor:actionDetermByPic} gives some restrictions on the automorphism group of such surfaces.

	\begin{ex}\label{ex:nonAutGeneral}
		Let $\zeta$ be a primitive 3-rd root of unity. 
		\begin{enumerate}
			\item Let $S\subset\PP^3$ be a smooth quartic surface given by a general equation of the form 
			\[F_4(x_0,x_1,x_2)+F_1(x_0,x_1,x_2)x_3^3=0,\] 
			where each $F_i$ is a homogeneous polynomial of degree $i$. By \cite[Theorem 3.3 and Proposition 4.9]{ArtebaniSarti}, $\rho(S)=4$.
			Consider the automorphism $g\in\Aut(S)$ defined by $g(x_0,x_1,x_2,x_3)=(x_0,x_1,x_2,\zeta x_3)$. 
			It has order three and acts on the 2-form $\omega_S$ by $g^*\omega_S=\zeta\omega_S$. 
			\item Let $S'\subset\PP^4$ be the complete intersection of a quadric and a cubic hypersurfaces given by general equations of the form:
			\[\left\{\begin{array}{l}
				F_2(x_0,...,x_3)=0 \\
				F_3(x_0,...,x_3)+x_4^3=0 \ ,
			\end{array}\right. \]
			where each $F_i$ is a homogeneous polynomial of degree $i$. By \cite[Theorem 3.3 and Proposition 4.7]{ArtebaniSarti}, $\rho(S')=2$.
			Consider the automorphism $g\in\Aut(S')$ defined by $g(x_0,x_1,x_2,x_3,x_4)=(x_0,x_1,x_2,x_3,\zeta x_4)$. 
			It has order three and acts on the 2-form $\omega_{S'}$ by $g^*\omega_{S'}=\zeta\omega_{S'}$.
		\end{enumerate}
	\end{ex}
	
	\subsection{K3 surfaces of Picard rank $2$}\label{sec:S_with_rho=2}
	In this subsection, we discuss K3 surfaces of Picard rank $2$ and their automorphism groups, with special attention to smooth quartic surfaces. 
	
	Let $S$ be a K3 surface with $\rho(S) = 2$. 
	By \cite[Proposition 3 and Proposition 4]{GLP}, the finite index subgroup $\Aut^\pm(S)$ of $\Aut(S)$ is isomorphic to one of the following groups: $\{id\}$, $\ZZ_2$, $\ZZ$ or $\ZZ_2\ast\ZZ_2$. Moreover, the finiteness of $\Aut^\pm(S)$, and so of $\Aut(S)$, is equivalent to the existence of a divisor class $D\in \Pic(S)$ with $D^2\in\{0,-2\}$.

	\begin{rem} \label{rem:kovacs}
		This characterization of the finiteness of $\Aut(S)$ for a K3 surface $S$ with $\rho(S) = 2$ admits the following geometric interpretation. 
		Denote by $R_1$ and $R_2$ the two extremal rays of the cone of curves $\NE(S)$ of $S$.
		By \cite[Corollary 2]{Kovacs}, one of the following conditions holds:
		\begin{enumerate}
			\item For $i=1,2$,  $R_i$ is spanned by the class of an irreducible curve $C_i$, with $C_i^2=0$ or $C_i^2=-2$.
			\item For $i=1,2$,  $R_i$ does not contain the class of any irreducible curve. 
		\end{enumerate}
		In case (1), any automorphism $g \in \Aut(S)$ acts on $\Pic(S)$ by either fixing or exchanging the classes $[C_1]$ and $[C_2]$, and
		so  $\Aut^\pm(S)$ is either $\{id\}$ or $\ZZ_2$. In case (2), $S$ does not contain any curve $C$ with $C^2=0$ or $-2$. 
	\end{rem}

	On the other hand, when $S$ is a smooth quartic surface, we have the following characterization of elements of finite order in $\Aut(S)$. 
	
	\begin{prop}[{\cite[Proposition 5.1.5]{Paiva}}, {\cite[Proposition 12, Remark 13]{paivaquedo}}]\label{prop:chracterization_of_involutions}
		Let  $S\subset \PP^3$ be a smooth quartic surface with $\rho(S) = 2$. Any non-trivial automorphism $g\in\Aut(S)$ of finite order is an involution. Moreover,
		there is a one-to-one correspondence between involutions in $\Aut(S)$ and ample divisor classes $A\in \Pic(S)$ with $A^2=2$. 
		
		Explicitly, given an involution $g\in \Aut(S)$, the associated ample class $A\in \Pic(S)$ is the generator of the rank one invariant sublattice $H^2(S,\ZZ)^g \defeq \left\{x\in H^2(S,\ZZ) \mid g^*x=x \right\} \subset \Pic(S)$.
		Conversely, given an ample class $A\in \Pic(S)$ with $A^2=2$, the linear system $|A|$ defines a double cover $S\to \PP^2$, and
		the associated involution $g\in \Aut(S)$ is the corresponding deck transformation.
		It acts on $H^2(S,\ZZ)$ as the reflection along the line generated by $A$:
		\[
		g^*\alpha = (A\cdot \alpha)A - \alpha.
		\]
	\end{prop}

	It should be noted that both the assumption on the Picard rank and on $S$ being a quartic surface in $\PP^3$ are essential in Proposition~\ref{prop:chracterization_of_involutions}, as Example \ref{ex:nonAutGeneral} demonstrates.

	\begin{cor}\label{cor:actionDetermByPic}
		Let $S\subset \PP^3$ be a smooth quartic surface with $\rho(S)=2$.
		\begin{enumerate}
			\item\label{it:actionDetermByPic1} If $\Aut^{\pm}(S)$ is a finite group, then $\Aut(S) = \Aut^{\pm}(S)$.
			\item\label{it:actionDetermByPic2} If two automorphisms $\chi, \tau\in \Aut(S)$ induce the same action on $\Pic(S)$, then
			$\chi = \tau$.
			\item\label{it:noRealizationByAutos} $\Aut(\PP^3;S) =\{id\}$.
		\end{enumerate}
	\end{cor}
	
	\begin{proof}
		Since $\Aut^{\pm}(S)$ is a subgroup of finite index in $\Aut(S)$, if
		$\Aut^{\pm}(S)$ is finite, then so is $\Aut(S)$.
		By Proposition \ref{prop:chracterization_of_involutions}, all non-trivial elements in $\Aut(S)$ are involutions, and so can only act on $\omega_S$ by multiplication by $\pm 1$.
		Therefore $\Aut(S) = \Aut^{\pm}(S)$.
		
		As for \eqref{it:actionDetermByPic2}, suppose that $\chi$ and $\tau$ induce the same action on $\Pic(S)$. 
		Write  $\chi^*\omega_S = \zeta_\chi \omega_S$ and $\tau^*\omega_S = \zeta_\tau\omega_S$ for suitable roots of unity $\zeta_\chi$ and $\zeta_\tau$.
		Then $f \defeq \chi \tau^{-1}$ acts trivially on the Picard group and by multiplication by $\zeta_{\chi}\zeta_{\tau}^{-1}$ on $\omega_S$, in particular it has finite order.        
		If $f$ were non-trivial, by Proposition \ref{prop:chracterization_of_involutions}, its invariant lattice would have rank one, a contradiction.
		Thus $f = id$ and consequently $\chi = \tau$.
		
		To prove \eqref{it:noRealizationByAutos}, let $g \in \Aut(\PP^3;S)$, and suppose that $g \neq id$.  
		Then $g$ has finite order by \cite[Theorem 1]{MM64}.
		By Proposition \ref{prop:chracterization_of_involutions}, $\varphi:=g|_S$ is an involution and its invariant lattice is generated by an ample divisor $A$, with $A^2 = 2$.
		Since $\varphi$ is the restriction of an automorphism of $\PP^3$ stabilizing $S$, it preserves the hyperplane class $H$ and so $H \in \langle A \rangle$. 
		Therefore 
		\[
		H = kA \implies 4 = H^2 = k^2A^2 = 2k^2,
		\]
		which is absurd. 
		We conclude that $g = id$.
	\end{proof}

	Putting \cite[Proposition 3 and Proposition 4]{GLP}, Proposition~\ref{prop:chracterization_of_involutions} and Corollary \ref{cor:actionDetermByPic} together yields the following:

	\begin{prop}
		\label{prop:possibilitiesForAut}
		Let $S\subset \PP^3$ be smooth quartic surface with $\rho(S) = 2$. 
		Then we have the following four possibilities for the subgroup $\Aut^\pm(S)$ of $\Aut(S)$:
		
		\setlength{\tabcolsep}{0pt}\renewcommand{\arraystretch}{1.1}
		\begin{tabular}{lll}
			$\Aut^\pm(S) = \{id\}$ & $\iff$ & \,$\mspace{2mu} \exists D\in \Pic(S)$ with $D^2\in\{0,-2\}$ and $\not\exists A\in \Pic(S)$ ample with $A^2=2$;
			\\
			$\Aut^\pm(S) = \ZZ_2$ & $\iff$ & \,$\mspace{2mu} \exists D\in \Pic(S)$ with $D^2\in\{0,-2\}$ and \,$\mspace{2mu}\exists A\in \Pic(S)$ ample with $A^2=2$;
			\\
			$\Aut^\pm(S) = \ZZ_2\ast\ZZ_2$ & $\iff$ &$\not\exists D\in \Pic(S)$ with $D^2\in\{0,-2\}$ and \,$\mspace{2mu}\exists A\in \Pic(S)$ ample with $A^2=2$;
			\\
			$\Aut^\pm(S) = \ZZ$ & $\iff$ & $\not\exists D\in \Pic(S)$ with $D^2\in\{0,-2\}$ and $\not\exists A\in \Pic(S)$ ample with $A^2=2$.
		\end{tabular}   \\

		Moreover, in the first two cases, $\Aut(S) = \Aut^{\pm}(S)$.
		
	\end{prop}

	Let $S\subset \PP^3$ be a smooth quartic surface with $\rho(S) = 2$, and denote by  
	$H$ the class of a hyperplane section of $S$. Since $H$ is a primitive element of $\Pic(S)$, 
	we can write the Picard lattice as $\Pic(S)=\ZZ H\oplus \ZZ W$ for some divisor class $W$. 
	In this basis, the intersection product is given by
	\begin{equation}\tag{$\star$}\label{matrixbilinearform}
		Q \ = \ \begin{pmatrix}
			4& b \\ 
			b & 2c
		\end{pmatrix}. 
	\end{equation}
	Notice that the discriminant $r = \disc(S) = b^2 - 8c$ is a positive integer, since $\Pic(S)$ has signature $(1,1)$. 
	Moreover, $r \equiv b^2 \equiv 0,1,4 \ (\modd{8})$.

	\begin{rem}\label{rmk:discriminant&Pell}
		Let $S\subset \PP^3$ be a smooth quartic surface with $\rho(S) = 2$ and discriminant $r$, and write 
		$\Pic(S)=\ZZ H\oplus \ZZ W$ as above. 
		The discriminant $r$ carries information about which curves exist on $S$.
		Indeed, given a divisor class $D=nH+mW$ on $S$, we have $4D^2=d^2-rm^2$, where $d=D\cdot H$. 
		Hence, the existence of a divisor $D$ on $S$ with $D^2=k$ is equivalent to the existence of an integer solution of the generalized Pell equation $x^2-ry^2=4k$.  
	\end{rem}

	In Section~\ref{section:main} we will investigate Gizatullin's problem for smooth quartic surfaces $S\subset \PP^3$ with $\rho(S) = 2$.
	We will reduce the problem to quartic surfaces  with discriminant $r \leq 57$ and $r\neq 52$ (Corollary~\ref{cor:noRealizationForr>57}). 
	The next proposition describes the automorphism group $\Aut(S)$ in these cases.

	\begin{prop}\label{prop:Autforrleq57}
		The sets
		\begin{itemize}
			\item[] $\mathcal{R}_0 = \{9, 12, 16, 24, 25, 33, 36, 44, 49, 57\}$,
			\item[] $\mathcal{R}_1 = \{17, 41\}$,
			\item[] $\mathcal{R}_2 = \{28, 56\}$, and 
			\item[] $\mathcal{R}_3 = \{20, 32, 40, 48\}$
		\end{itemize}
		give a partition of all integers $r \leq 57$, $r\neq 52$, such that $r$ is the discriminant of a smooth quartic surface $S\subset \PP^3$ with $\rho(S) = 2$.
		
		The subgroup $\Aut^\pm(S)$ of $\Aut(S)$ is described as follows:
		\[
		\Aut^\pm(S) \cong 
		\left\{
		\begin{array}{ll}
			\{id\},            & \text{ if } r \in \mathcal{R}_0;\\
			\ZZ_2,            & \text{ if } r \in \mathcal{R}_1;\\
			\ZZ_2 \ast \ZZ_2, & \text{ if } r \in \mathcal{R}_2; \\
			\ZZ,              & \text{ if } r \in \mathcal{R}_3.    
		\end{array}   
		\right.
		\]
		Moreover, in the first two cases, $\Aut(S) = \Aut^{\pm}(S)$.
	\end{prop}
	
	\begin{proof}
		As before, we denote by $H$ be the class of a hyperplane section of $S$ and let $\{H,W\}$ be a basis of $\Pic(S)$. 
		With respect to this basis, the intersection product in $\Pic(S)$ is given by the matrix \eqref{matrixbilinearform}, 
		and $r  \equiv 0,1,4 \ (\modd{8})$. 
		If $r\in\{1,4,8\}$, then we can find an irreducible curve $E$ such that 
		$(E^2, H\cdot E)\in \{(-2,0),(0,1),(0,2)\}$, which contradicts the fact that $H$ is very ample (Proposition \ref{prop:criterioveryample}). 
		Conversely, by Theorem \ref{thm:K3andlattices}, Proposition \ref{prop:criterioveryample} and \cite[Theorem 6.1]{SD74}, every even lattice with bilinear form given by \eqref{matrixbilinearform}, signature $(1,1)$ and discriminant $8<r\leq 57$ can be realized as the Picard lattice of a smooth quartic surface. 
		This proves the first assertion. 
		
		By Proposition \ref{prop:possibilitiesForAut}, the subgroup $\Aut^\pm(S)$ of $\Aut(S)$ is completely determined by the existence of a divisor $D$ with $D^2\in\{0,-2\}$ and an ample divisor $A$ with $A^2=2$. 
		By Remark~\ref{rmk:discriminant&Pell}, the existence of a divisor $\Delta$ on $S$ with $\Delta^2=k$, is equivalent to the existence of an integer solution of the generalized Pell equation $x^2-ry^2=4k$. 
		The second assertion then follows from checking the existence of integer solutions of the corresponding generalized Pell equations for each value of $r\in \mathcal{R}_i$,  $i\in\{0,1,2,3\}$.

		For each $r\in\mathcal{R}_0\cup \mathcal{R}_1$, either the equation $x^2-ry^2=0$ or the equation $x^2-ry^2=-8$ has an integer solution, as illustrated in the following table. This implies that $\Aut^\pm(S)= \{id\}$ or $\Aut^\pm(S)\cong \ZZ_2$.
		\[
		\arraycolsep=5pt \def\arraystretch{1.2}
		\begin{array}{|c|c|c|c|c|c|c|c|c|c|c|c|c|}
			\rowcolor{gray!25}
			\hline
			r      & 9      & 12    & 16  & 17   & 24    & 25      & 33    & 36  & 41  & 44 & 49 &57    \\
			\hline
			(x,y)  & (1,1) & (2,1) & (4,1) & (3,1) & (4,1) & (5,1) & (5,1) & (6,1) & (19,3) & (6,1) & (7,1) & (7,1) \\[1pt]
			\hline  \rowcolor{gray!7}
			x^2-ry^2& -8 & -8 & 0 & -8 & -8 & 0 & -8 & 0 & -8 & -8 & 0 & -8\\
			\hline
		\end{array}
		\]
		
		Suppose that $r\in\mathcal{R}_0 = \{9, 12, 16, 24, 25, 33, 36, 44, 49, 57\}$.
		In order to show that $\Aut^\pm(S)= \{id\}$, we will show that there are no divisors with square $2$, or equivalently that $x^2-ry^2=8$ does not have integer solutions.
		If $r\in\{9,12,24,33,36,44,57\}$, then either $r\equiv 0\ (\modd{3})$ or $r\equiv 0\ (\modd{11})$. So the equation $x^2-ry^2=8$ reduces to  either $x^2\equiv 2\ (\modd{3})$ or $x^2\equiv 8\ (\modd{11})$, and one checks easily that these have no integer solution. 
		If $r\in\{16,25,49\}$, then $r=t^2$ for an appropriate integer $t>1$.
		We set $z=ty$ and rewrite the equation $x^2-ry^2=8$ as $x^2-z^2=8$. An integer solution $(x,z)$ must satisfy $x^2>z^2>1$.
		Then, from $$8=x^2-z^2=|x|^2-|z|^2\geq |x|^2-(|x|-1)^2=2|x|-1,$$
		we conclude that $2\leq |z|<|x|\leq4$, and one checks easily that there are no integer solutions.

		Suppose that $r\in\mathcal{R}_1 = \{17, 41\}$. If $r=b^2 - 8c=17$, then $(x,y)=(5,1)$ and $(5,-1)$ are solutions of $x^2-17y^2=8$ and one of them satisfies that $z\defeq\frac{x-yb}{4}\in\ZZ$. 
		For such pair $(x,y)$,  $A=zH+yW\in \Pic(S)$ is the corresponding divisor on $S$ with $A^2=2$. 
		Note that $A\cdot H=5$, $\Pic(S)=\langle H,A\rangle$, and $A$ is effective. By Corollary \ref{cor:linearSystem}, $A$ is nef (and big).
		To show that $A$ is ample, it is enough to check that there is no rational curve $\Gamma$ such that $A\cdot \Gamma=0$. 
		Indeed, if there is such a curve $\Gamma$, then $E=A+\Gamma\in \Pic(S)$ satisfies $E^2=0$, and so $0=4E^2=(H\cdot E)^2-17m^2$, where $m\in\ZZ$ is such that $E=nH+mW$ in $\Pic(S)$.  This is not possible since $r=17$ is not a square number. 
		When $r=41$, we argue in the same way, with $(x,y)=(7,1),(7,-1)$ being solutions of  $x^2-41y^2=8$. 
		
		If $r\in\mathcal{R}_2= \{28, 56\}$, then $r\equiv0\ (\modd{7})$. So the equation $x^2-ry^2=-8$ reduces to $x^2\equiv 6\ (\modd{7})$, which does not have solutions. Together with the fact that $r$ is not a square number, this implies the non-existence of divisors on $S$ with self-intersection $0$ or $-2$. 
		In order to show that $\Aut^\pm(S)\cong \ZZ_2 \ast \ZZ_2$, we must verify the existence of $A\in \Pic(S)$ ample with $A^2=2$. 
		Indeed, the pair $(x,y)=(6,1)$ (respectively $(x,y)=(8,1)$) is a solution of the equation $x^2-ry^2=8$ for $r=28$ (respectively $r=56$), 
		and the corresponding divisor $A$ is automatically ample since $S$ has no rational curves.

		Finally, suppose that  $r\in\mathcal{R}_3 = \{20, 32, 40, 48\}$. Since $r$ is not a square number, $x^2-ry^2=0$ does not have integer solutions. If $r\in\{20,40\}$, then $r\equiv 0\ (\modd{5})$. So the equations $x^2-ry^2=-8$ and $x^2-ry^2=8$ reduce to $x^2\equiv3$ and $x^2\equiv2 \ (\modd{5})$, none of which has solutions. 
		If $r\in\{32,48\}$, then write $r=16s$ for the appropriate integer $s\in\{2,3\}$. If $(x,y)$ is a solution of $x^2-16sy^2=-8$ or $x^2-16sy^2=8$, then $x=2z$ is an even integer. So these equations can be simplified to $z^2-4sy^2=-2$ and $z^2-4sy^2=2$, and then reduced to $z^2\equiv 2\ (\modd{4})$, which does not have solutions.  
		
		The last claim follows from Corollary \ref{cor:actionDetermByPic}\eqref{it:actionDetermByPic1}.
	\end{proof}

	In Section~\ref{section:main}, in order to realize elements of $\Aut^{\pm}(S)$
	for a quartic surface $S\subset \PP^3$ with $\rho(S) = 2$ and discriminant $r\in \mathcal{R}_1 \cup \mathcal{R}_2\cup \mathcal{R}_3$ as restrictions of Cremona transformations, we will need to write down explicit generators for $\Aut^{\pm}(S)$. 
	To do so, we will use the following results from \cite{Lee}. 
	In what follows, we denote by $H$ the class of a hyperplane section of $S$, extend it to a basis $\{H,W\}$ of $\Pic(S)$, and write the intersection matrix $Q$ in this basis as in \eqref{matrixbilinearform}.

	\begin{prop}[{\cite[Theorem 1.1 and Lemma 2.6]{Lee}, \cite[Lemma 2.1.19]{Paiva}}] \label{prop:Leeinfiniteordergen} \label{prop:Leeinvolutionsgen}
		Let $S\subset \PP^3$ be a smooth quartic surface with $\rho(S)=2$, and let $H$ and $Q$ be as above. 
		
		An isometry $\phi\in O(\Pic(S))$ is induced by an automorphism $g\in \Aut^{\pm}(S)$ if and only if 
		\[
		(\phi+ Id)*Q^{-1}\in M_{2\times 2}(\ZZ) \text{ or }  (\phi- Id)*Q^{-1}\in M_{2\times 2}(\ZZ),  \text{ and } \phi(H) \text{ is ample. }
		\]
		
		Furthermore, we have the following characterizations of involutions and automorphisms of infinite order.
		\begin{enumerate}
			\item \label{Leeinvolution} 
			The automorphism $g$ is an involution if and only if the corresponding isometry $\phi=g^*$ is of the form 
			\begin{equation*}
				\phi=\begin{pmatrix}
					\alpha & \beta \\
					-\frac{b}{c}\alpha+\frac{2}{c}\beta & -\alpha
				\end{pmatrix},
			\end{equation*}
			where $(\alpha,\beta)$ is an integer solution of the equation:
			\begin{equation}\tag{$\ast$}\label{quadeq}
				\alpha^2-\frac{b}{c}\alpha\beta+\frac{2}{c}\beta^2=1.
			\end{equation}
			
			\item \label{Leeinfinite} The automorphism $g$ has infinite order if and only if the corresponding isometry $\phi=g^*$ is of the form 
			\begin{equation*}    
				\phi=\begin{pmatrix}
					\alpha & \beta \\
					-\frac{2}{c}\beta & \alpha-\frac{b}{c}\beta 
				\end{pmatrix},
			\end{equation*}
			where $(\alpha,\beta)$ is an integer solution of the equation \emph{(\ref{quadeq})}.
			In this case, $\phi$ is a power of
			\begin{equation*}    
				h= \begin{pmatrix}
					\alpha_1 & \beta_1 \\
					-\frac{2}{c}\beta_1 & \alpha_1-\frac{b}{c}\beta_1 
				\end{pmatrix},
			\end{equation*}
			where $(\alpha_1,\beta_1)$ is a minimal positive integer solution of \emph{(\ref{quadeq})}.  
		\end{enumerate}
	\end{prop}

	\begin{rem}
		The necessary and sufficient conditions for an isometry $\phi\in O(\Pic(S))$ to be induced by an automorphism of $S$ in Proposition~\ref{prop:Leeinvolutionsgen} are exactly the Gluing and Torelli conditions from Theorems~\ref{thm:gluing} and \ref{thm:torelli}.
		More precisely $(\phi\pm Id)*Q^{-1}\in M_{2\times 2}(\ZZ)$ if and only if $\overline{\phi}=\mp Id$ on $A(\Pic(S))$. By Theorem~\ref{thm:gluing}, this is equivalent to the existence of an isometry $\Phi$ on $H^2(S,\ZZ)$ whose restrictions to $\Pic(S)$ and $T(S)$ are $\phi$ and $\mp Id$, respectively. This isometry $\Phi$ is automatically a Hodge isometry, so, by Theorem~\ref{thm:torelli}, it is induced by an automorphism of $S$ if and only if  $\phi(H)$ is ample.         
		The assumption that $g\in \Aut^{\pm}(S)$ ensures that the corresponding isometry acts as $\pm Id$ on $T(S)$. 
		
	\end{rem}

	\section{Ingredients from Birational Geometry}\label{section:birational}

	\subsection{The Sarkisov Program}
	
	Given a uniruled variety, the MMP produces a \emph{Mori fiber space} that is birationally equivalent to it. 
	In general, there are several different Mori fiber spaces in the same birational equivalence class. The  \emph{Sarkisov program} provides a factorization theorem for  birational maps between Mori fiber spaces in terms of simpler birational maps, called  \emph{Sarkisov links}. It was established in dimension~3 in~\cite{Corti}, and in higher dimensions in~\cite{HM13}.
	We start this section by fixing notation and recalling some basic notions of the MMP and the Sarkisov Program. 
	
	\begin{defnot} \leavevmode \label{def:contr}
		\begin{enumerate} 
			\item  A \emph{divisorial (extremal) contraction} is a birational contraction $f\colon Z \to X$ of relative Picard rank $1$ between $\QQ$-factorial terminal varieties associated to an extremal ray $R\subset \NE(Z)$ such that $K_Z\cdot R<0$. It contracts a divisor of $Z$ onto a subvariety of codimension $\geq 2$ in $X$, which is called the \emph{center} of the divisorial contraction. 
			
			\item\label{item:flip} Let $\varphi\colon Z \psmap Z^\prime$ be a small birational map between $\QQ$-factorial terminal varieties fitting into a commutative diagram
			\[
			\xymatrix@R=.5cm@C=.5cm{
				Z \ar@{..>}[rr]^{\varphi} \ar[rd]_{s} && Z^\prime \ar[ld]^(.43){s^\prime},\\
				& W
			}
			\]
			where $s\colon Z \to W$ and $s^\prime\colon Z^\prime \to W$ are small contractions of relative Picard rank $1$ associated to extremal rays  $R\subset \NE(Z)$ and $R^\prime\subset \NE(Z^\prime)$. 
			We say that $\varphi$ is a \emph{flip} if $K_Z\cdot R<0$ and $K_{Z^\prime}\cdot R^\prime>0$. 
			We say that $\varphi$ is an \emph{antiflip} if $K_Z\cdot R>0$ and $K_{Z^\prime}\cdot R^\prime<0$. 
			We say that $\varphi$ is a \emph{flop} if $K_Z\cdot R=0=K_{Z^\prime}\cdot R^\prime$.
			\item  A \emph{Mori fiber space} is a morphism $X \to B$ with connected fibers and relative Picard rank $1$  from a $\QQ$-factorial terminal variety $X$ onto a lower dimensional normal variety $B$ associated to an extremal ray $R\subset \NE(X)$ such that $K_X\cdot R<0$.
		\end{enumerate}
	\end{defnot}
	
	The following result is probably well known to experts. We include a proof due to the lack of a reference. 	
	
	\begin{lem}
		\label{lem:antiflipsRationalCurves}
		Let $\varphi\colon Z \psmap Z^\prime$ be a small birational map as in Definition \ref{def:contr}\eqref{item:flip}.
		\begin{enumerate}
			\item\label{item:everythingIsAFlop} There exist divisors $\Delta$ and $\Delta'$ in $\Pic(Z)_{\QQ}$ and $\Pic(Z')_{\QQ}$ respectively, such that
			\[
			(K_Z + \Delta)\cdot R = (K_{Z'} + \Delta')\cdot R' = 0,
			\]
			and both $(Z,\Delta)$ and $(Z',\Delta')$ are klt.
			\item\label{item:generationByRationalCurves} The rays $R$ and $R'$ are both generated by rational curves.
			\item\label{item:generationBySmoothRationalCurves} If moreover $\dim(Z) = \dim(Z') = 3$, then every irreducible component of $\Exc(s)$ and $\Exc(s')$ is isomorphic to $\PP^1$.
		\end{enumerate}
	\end{lem}

	\begin{proof}
		Let $Z, Z',\varphi, s$ and $s'$ be as in Definition \ref{def:contr}\eqref{item:flip}.
		If $\varphi$ is a flop, then we simply choose $\Delta = \Delta' = 0$.
		Without loss of generality, assume that $\varphi$ is a flip, i.e.,  $K_Z\cdot R<0$.
		We take $A$ a general and sufficiently (very) ample divisor on $Z$ and $\Delta = t_AA$, $0  <t_A <1$, so that the pair $(Z,\Delta)$ is klt and $(K_Z+\Delta)\cdot R=0$. Therefore, $K_Z + \Delta$ is $\QQ$-linearly equivalent to the pullback of a $\QQ$-Cartier $\QQ$-divisor on $W$, namely $K_Z + \Delta = s^*(K_W + \Delta_W)$, where $\Delta_W = s_* \Delta$. Since $(Z,\Delta)$ is klt, so are $(W,\Delta_W)$ and $(Z',\Delta')$, where $\Delta'=\varphi_*\Delta$ and $K_{Z'} + \Delta' = {s'}^*(K_W + \Delta_W)$. This proves \eqref{item:everythingIsAFlop}.

		To prove \eqref{item:generationByRationalCurves}, let $\Delta$ be as in \eqref{item:everythingIsAFlop}, $A'$ an effective ample divisor on $Z'$, and $A$ its strict transform on $Z$.
		Then $A$ is $s$-antiample.
		In particular, for any $t > 0$ sufficiently small, $(Z,\Delta + tA)$ is klt and $(K_Z + \Delta + tA)\cdot R < 0$.
		We conclude that $R$ is generated by a rational curves by \cite[Theorem 3.7(1)]{KM98}, and similarly for $R'$.

		By \eqref{item:everythingIsAFlop} $(W,\Delta_W)$ is klt, and so, by \cite[Theorem 5.22]{KM98}, $W$ has rational singularities.
		By taking a resolution of singularities $f:X \to Z$ and considering the Grothendieck spectral sequence for $f_*$ and $s_*$, one checks that $R^is_*(\OO_Z) =  0$  $\forall i>0$, and then \eqref{item:generationBySmoothRationalCurves} follows from \cite[Lemma 3.4]{Kaw88}.
	\end{proof}

	Next we recall the definition of the four types of Sarkisov links between Mori fiber spaces.

	\begin{defi}[Sarkisov links]\label{def:links}
		In the following diagrams, $X\to B$ and  $X^\prime\to B^\prime$ denote Mori fiber spaces. 
		\begin{enumerate}[(I)]
			\item A \emph{Sarkisov diagram of type~(I)} is a commutative diagram
			\[
			\begin{tikzpicture}[xscale=1.5,yscale=-1.2]
				\node (A0_0) at (0, 0) {};
				\node (A0_1) at (1, 0) {$Z$};
				\node (A0_2) at (2, 0) {$X^\prime$};
				\node (A1_0) at (0, .8) {$X$};
				\node (A1_2) at (2, .8) {$B^\prime$};
				\node (A2_0) at (0, 2.3) {$B$};
				\path (A1_2) edge [->]node [auto] {$\scriptstyle{}$} (A2_0);
				\path (A0_1) edge [->]node [auto] {$\scriptstyle{}$} (A1_0);
				\path (A0_1) edge [->,dotted]node [auto] {$\scriptstyle{}$} (A0_2);
				\path (A0_2) edge [->]node [auto] {$\scriptstyle{}$} (A1_2);
				\path (A1_0) edge [->]node [auto] {$\scriptstyle{}$} (A2_0);
			\end{tikzpicture}
			\]
			where $Z\to X$ is a divisorial contraction and
			$Z\psmap X^\prime$ is a sequence of flips, flops and antiflips.
			The map $X\dasharrow X^\prime$ is called a \emph{Sarkisov link of type~(I)}.

			\item A \emph{Sarkisov diagram of type~(II)} is a commutative diagram
			\[
			\begin{tikzpicture}[xscale=1.5,yscale=-1.2]
				\node (A0_1) at (1, 0) {$Z$};
				\node (A0_2) at (2, 0) {$Z^\prime$};
				\node (A1_0) at (0, 1) {$X$};
				\node (A1_3) at (3, 1) {$X^\prime$};
				\node (A2_0) at (0, 2) {$B$};
				\node (A2_3) at (3, 2) {$B$};
				\path (A0_1) edge [->]node [auto] {$\scriptstyle{}$} (A1_0);
				\path (A1_3) edge [->]node [auto] {$\scriptstyle{}$} (A2_3);
				\path (A2_0) edge [-,double distance=1.5pt]node [auto] {$\scriptstyle{}$} (A2_3);
				\path (A1_0) edge [->]node [auto] {$\scriptstyle{}$} (A2_0);
				\path (A0_2) edge [->]node [auto] {$\scriptstyle{}$} (A1_3);
				\path (A0_1) edge [->,dotted]node [auto] {$\scriptstyle{}$} (A0_2);
			\end{tikzpicture}
			\]
			where $Z\to X$ and $Z^\prime \to X^\prime$ are divisorial
			contractions and $Z\psmap Z^\prime$ is a sequence of flips, flops and antiflips.
			In order to avoid trivial diagrams, we also require that the common relative effective cone of $Z$ and $Z^\prime$ over $B$ be generated by the exceptional divisors of $Z\to X$ and $Z^\prime \to X^\prime$.
			The map $X\dasharrow X^\prime$ is called a \emph{Sarkisov link of type~(II)}.

			\item A \emph{Sarkisov link of type~(III)} is the inverse of a Sarkisov link of type~(I).
			\item A \emph{Sarkisov diagram of type~(IV)} is a commutative diagram
			\[
			\begin{tikzpicture}[xscale=1.5,yscale=-1.2]
				\node (A0_0) at (0, 0) {$X$};
				\node (A0_2) at (2, 0) {$X^\prime$};
				\node (A1_0) at (0, 1) {$B$};
				\node (A1_2) at (2, 1) {$B^\prime$};
				\node (A2_1) at (1, 2) {$T$};
				\path (A1_2) edge [->]node [auto] {$\scriptstyle{}$} (A2_1);
				\path (A0_0) edge [->]node [auto] {$\scriptstyle{}$} (A1_0);
				\path (A1_0) edge [->]node [auto] {$\scriptstyle{}$} (A2_1);
				\path (A0_2) edge [->]node [auto] {$\scriptstyle{}$} (A1_2);
				\path (A0_0) edge [->,dotted]node [auto] {$\scriptstyle{}$} (A0_2);
			\end{tikzpicture}
			\]
			where $X\psmap X^\prime$ is a sequence of flips, flops and
			antiflips, and $B\rightarrow T$ and $B^\prime\rightarrow T$ are Mori contractions.
			In order to avoid trivial diagrams, we also require that the common relative effective cone of $X$ and $X^\prime$ over $T$ be generated by the pullbacks to $X$ and $X^\prime$ of ample divisors on $B$ and $B^\prime$, respectively.
			The map $X\psmap X^\prime$ is called a \emph{Sarkisov link of type~(IV)}.
		\end{enumerate}
		
		In the context of a Sarkisov diagram of type~(I) or (II) above, we say that the divisorial contraction $Z\to X$ \emph{initiates the Sarkisov link}.
	\end{defi}

	\begin{thm}[The Sarkisov Program - \cite{Corti}, \cite{HM13}]\label{thm:SarkisovProgram}
		Every birational map between Mori fiber spaces can be factorized as a composition of Sarkisov links.
	\end{thm}
	
	Even though not strictly necessary, we find it useful to adopt the point of view of \cite{BLZ}, in which Sarkisov links correspond to rank $2$ fibrations. 
	We refer to \cite[Definition 2.2.]{BLZ} for the definition of relative Mori Dream Space. 
	
	\begin{defi}[{\cite[Definition 3.1]{BLZ}}]\label{def:rankRFibrations}
		Let $r$ be an integer.
		A morphism $\eta\colon X \to B$ is a \emph{rank $r$ fibration} if the following conditions hold:
		\begin{enumerate}
			\item\label{item:RRF1} $X/B$ is a relative Mori Dream Space.
			\item\label{item:RRF2} $\dim(X) > \dim(B)$ and $\rho(X/B) = r$.
			\item\label{item:RRF3} $X$ is $\QQ$-factorial and terminal, and for any divisor $D$ on $X$, the output of any $D$-MMP from $X$ over $B$ is still $\QQ$-factorial and terminal.
			\item\label{item:RRF4} There exists an effective $\QQ$-divisor $\Delta_B$ on $B$ such that $(B, \Delta_B)$ is klt.
			\item\label{item:RRF5} $-K_X$ is $\eta$-big.
		\end{enumerate}
	\end{defi}
	
	The notion of rank $r$ fibrations encompasses the notions of Mori fiber spaces and Sarkisov links.
	In particular, a rank $1$ fibration is a Mori fiber space (\cite[Lemma 3.3]{BLZ}), while rank $2$ fibrations correspond to Sarkisov links (\cite[Lemma 3.7]{BLZ}).
	In our approach to Gizatullin's problem, we will need to know when the blowup of a curve, contained in a quartic surface in $\PP^3$, initiates a Sarkisov link. In order to classify these curves when the quartic surface has Picard rank $2$, we will use the following criterion, which is a special case of \cite[Lemma 3.7]{BLZ}. 
	
	\begin{lem} \label{lemma:rank2fibration}
		Let $C \subset \PP^3$ be a curve, and let $X$ denote the blowup of $\PP^3$ along $C$.
		Then $X \to \PP^3$  initiates a Sarkisov link if and only if $X \to \PP^3 \to \Spec(\CC)$ is a rank $2$ fibration.
	\end{lem}

	In \cite[Theorem 1.1]{BL}, Blanc and Lamy classified smooth curves $C \subset \PP^3$ whose blowups $X$ are \emph{weak Fano}, i.e., 
	$-K_X$ is nef and big. 
	The following observation allows one to check which of these blowups give rank $2$ fibrations.
	
	\begin{lem} \label{lemma:weak-fano&rank2fibration}
		Let $C \subset \PP^3$ be a  curve, and suppose that  the blowup $X$ of $\PP^3$ along $C$ is terminal and weak Fano. 
		Then $X \to \PP^3 \to \Spec(\CC)$ is a rank $2$ fibration if and only if the morphism to the anti-canonical model of $X$ is a small contraction.
	\end{lem}
	
	\begin{proof}
		Since $X$ is weak Fano, it is a Mori Dream Space and $-K_X$ is semi-ample.
		So conditions \eqref{item:RRF1}, \eqref{item:RRF2}, \eqref{item:RRF4} and \eqref{item:RRF5} in Definition \ref{def:rankRFibrations} are all satisfied. 
		If $X$ is Fano, then any $D$-MMP is also a $(K_{X})$-MMP. Thus, the output of any $D$-MMP is terminal, and so
		$X \to \PP^3 \to \Spec(\CC)$ is a rank $2$ fibration.
		From now on, suppose that $X$ is strictly weak Fano.  
		The morphism $X \to \check{X}$ to the anti-canonical model of $X$ is either a divisorial contraction or a small contraction.
		When $X \to \check{X}$ is divisorial, $\check{X}$ has worse than terminal singularities, and therefore $X \to \PP^3 \to \Spec(\CC)$ is \emph{not} a rank $2$ fibration, as it violates condition \eqref{item:RRF3} in Definition \ref{def:rankRFibrations}. 
		Suppose that $X \to \check{X}$ is a small contraction, and 
		consider its flop 
		\[
		\xymatrix@R=.5cm@C=.5cm{
			X \ar[rd] \ar@{..>}[rr] && X^+ \ar[ld]\\
			& \check{X}.
		}
		\]
		Note that $-K_{X^+}$ is the pullback of $-K_{\check{X}}$, which is ample, and so $X^+$ is again terminal and weak Fano.
		Write $R_1^+$ and $R_2^+$ for the two extremal rays of $\NE(X^+)$, where $R_1^+$ corresponds to $X^+ \to \check{X}$ and $R_2^+$ is $K_{X^+}$-negative.
		Let $D$ be any divisor on $X$.
		Then any $D$-MMP either ends with $\PP^3$, or it factors through $X^+$.
		In the latter case, any further step is associated to the contraction of $R_2^+$, and is therefore a $(K_{X^+})$-MMP too.
		In any case, the output of any $D$-MMP is terminal, and so
		$X \to \PP^3 \to \Spec(\CC)$ is a rank $2$ fibration.
	\end{proof}
	
	\begin{rem}\label{rem:quadrisecants}
		Let the notation and assumptions be as in Lemma~\ref{lemma:weak-fano&rank2fibration}.
		The $K_{X}$-trivial curves on $X$ are precisely the strict transforms of 
		curves of degree $\delta$ on $\PP^3$ that are $4\delta$-secant to $C$.
		Therefore, the condition that the morphism to the anti-canonical model of $X$ is small is equivalent to the finiteness of such secant curves. 
	\end{rem}	
	
	The possibilities for the genus and degree of smooth curves $C \subset \PP^3$ whose blowups $X$ are weak Fano are listed in \cite[Table 1]{BL}, as well as whether or not the morphism $X \to \check{X}$ to the anti-canonical model is divisorial for a general curve in the corresponding Hilbert scheme. 
	So, by putting together \cite[Theorem 1.1 and Table 1]{BL} and Lemma~\ref{lemma:weak-fano&rank2fibration}, we get the following classification. 
	
	\begin{thm}\label{thm:blancLamyCurves}
		Let $C \subset \PP^3$ be a smooth curve of genus $g$ and degree $d$, and let $X$ denote the blowup of $\PP^3$ along $C$.   
		Suppose that $X$ is weak Fano and $X \to \PP^3 \to \Spec(\CC)$ is a rank $2$ fibration.
		Then 
		\begin{equation}\tag{$\dagger$}\label{list}
			(g,d) \in 
			\left\{
			\def\arraystretch{1.2}
			\begin{array}{c}
				\arraycolsep=1.4pt\def\arraystretch{1.2}
				\begin{array}{llllllllllll}
					(0,1), & (0,2), & (0,3), & (0,4), & (0,5), & (0,6), & (0,7), & (1,3), & (1,4), & (1,5), & (1,6), & (1,7),\\
					(2,5), & (2,6), & (2,7), & \mathbf{(2,8)}, & \mathbf{(3,6)}, & (3,7), &
					\mathbf{(3,8)}, & (4,6), & (4,7), & \mathbf{(4,8)}, & (5,7), & \mathbf{(5,8)}, 
				\end{array}\\
				\arraycolsep=2.6pt
				\begin{array}{llllllllll}
					{(6,8)}, & \mathbf{(6,9)}, & (7,8), & (7,9), & (8,9), & (9,9), & (10,9), & \mathbf{(10,10)}, & \mathbf{(11,10)}, & \mathbf{(14,11)}
				\end{array}
			\end{array}
			\right\}.
		\end{equation}
		
		Conversely, suppose that $(g,d) \in \eqref{list}$ and the smooth curve $C$ satisfies the following conditions, which define an open subset of the Hilbert scheme $\mathcal{H}_{g,d}$ of curves of arithmetic genus $g$ and degree $d$:
		\begin{enumerate}
			\item\label{genCond1} $C$ does not admit $5$-secant lines, $9$-secant conics, nor $13$-secant twisted cubics;
			\item\label{genCond2} there are finitely many irreducible curves in $X$ intersecting $-K_X$ trivially.
		\end{enumerate}
		Then $X$ is weak Fano and $X \to \Spec(\CC)$ is a rank $2$ fibration.
	\end{thm}
	
	The pairs $\mathbf{(g,d)}$ in bold are the ones that will be relevant to our approach to Gizatullin's problem in section \ref{section:main}.

	\subsection{Calabi-Yau pairs and volume preserving birational maps}\label{Subsection:CYpairs}
	
	We now introduce some basic notions in log Calabi-Yau geometry, and explain how to use the framework developed in \cite{CK16} and \cite{ACM} to investigate Gizatullin's problem. 
	
	\begin{defi} \label{def:CY}
		\begin{enumerate}
			\item A \emph{Calabi-Yau pair} is a pair $(X,D)$ consisting of a terminal projective $\QQ$-factorial variety $X$ and an
			effective Weil divisor $D$ on $X$ such that $K_X+D\sim 0$ and $(X,D)$ has log canonical singularities. We say that a  Calabi-Yau pair $(X,D)$ is \emph{canonical} if it has canonical singularities.
			\item A \emph{Mori fibered Calabi-Yau pair} is a Calabi-Yau pair $(X,D)$, together with a Mori fiber space structure $X \to B$.
			\item Let $(X,D_X)$ and $(Y,D_Y)$ be Calabi-Yau pairs, and $f\colon X\dasharrow Y$ a birational map, inducing an identification of the function fields $\CC(X)\cong_\CC\CC(Y)$.  We say that $f$ is \emph{volume preserving with respect to $(X,D_X)$ and $(Y,D_Y)$} if, for every divisorial valuation $E$ of $\CC(X)\cong_\CC\CC(Y)$, the discrepancies of $E$ with respect to the pairs $(X,D_X)$ and $(Y,D_Y)$ are equal: $a(E,X,D_X)=a(E,Y,D_Y)$. We refer to \cite[Definition 2.25]{KM98} for the notion of discrepancy.
			\item Given a Calabi-Yau pair $(X,D)$, we denote by $\Bir^{v.p.}(X,D)$ the group of birational self-maps of $X$ which are volume preserving with respect to $(X,D)$.
		\end{enumerate}
	\end{defi}

	\begin{rem}
		The \emph{volume preserving} terminology is explained by the following characterization (see \cite[Remark 1.7]{CK16}). Given Calabi-Yau pairs $(X,D_X)$ and $(Y,D_Y)$, there are rational volume forms $\omega_X$ on $X$ and $\omega_Y$ on $Y$, unique up to scaling, such that $D_X+\divv(\omega_X)=0$ and $D_Y+\divv(\omega_Y)=0$. A birational map $f\colon X\dasharrow Y$ induces an identification of the spaces of rational volume forms on $X$ and $Y$. It is volume preserving with respect to $(X,D_X)$ and $(Y,D_Y)$ if and only if it identifies the rational volume forms $\omega_X$ and $\omega_Y$, up to scaling.  
	\end{rem}

	\begin{rem} \label{rem:decomposition_group}
		For \emph{canonical} Calabi-Yau pairs, the volume preserving condition has the following simple interpretation 
		(see \cite[Proposition 2.6]{ACM}).  
		Let $(X,D_X)$ and $(Y,D_Y)$ be canonical Calabi-Yau pairs, and $f\colon X\dasharrow Y$ a birational map. 
		Then $f\colon X\dasharrow Y$ is volume preserving with respect to $(X,D_X)$ and $(Y,D_Y)$ if and only if it restricts to a birational map between $D_X$ and $D_Y$. 
		In particular, when $S\subset \PP^3$ is a smooth quartic surface, the pair $(\PP^3,S)$ is canonical and so the group of Cremona transformations stabilizing $S$ coincides with the group of volume preserving Cremona transformations with respect to $(\PP^3,S)$:
		\[
		\Bir\left(\PP^3;S\right) \ = \ \Bir^{v.p.}\left(\PP^3,S\right).
		\]
	\end{rem}

	An important tool to study volume preserving birational maps between Calabi-Yau pairs is the volume preserving variant of the Sarkisov program established in \cite{CK16}. 
	Before we state it, we introduce the volume preserving version of Sarkisov links. 
	
	\begin{defi}[{\cite[Definition 1.12 and Remark 1.13]{CK16}}] \label{def:vp_link} A \emph{volume preserving Sarkisov link} is a Sarkisov link as in Definition~\ref{def:links}, with the following additional data and property. 
		\begin{itemize}
			\item There are effective Weil divisors $D_X$ on $X$, $D_{X'}$ on $X'$, $D_Z$ on $Z$, and $D_{Z'}$ on $Z'$, making the pairs $(X,D_X)$, $(X',D_{X'})$, $(Z,D_Z)$ and $(Z',D_{Z'})$ Calabi-Yau pairs.
			\item All the divisorial contractions, flips, flops and antiflips that constitute the Sarkisov link are volume preserving for these Calabi-Yau pairs.
		\end{itemize}
	\end{defi}
	
	\begin{thm}[Volume preserving Sarkisov Program - \cite{CK16}]\label{thm:vp_SarkisovProgram}
		Every volume preserving birational map between Mori fibered Calabi-Yau pairs can be factorized as a composition of volume preserving Sarkisov links.
	\end{thm}
	
	The volume preserving condition imposes strong restrictions on the links appearing in a Sarkisov decomposition. 
	For example, in the context of Gizatullin's problem, we have the following. 
	
	\begin{prop} \label{CinS}
		Let $S\subset \PP^3$ be a smooth quartic surface, and  $f\colon (X,D_X) \to (\PP^3,S)$  a volume preserving divisorial contraction.  Then $f\colon X \to \PP^3$ is the blowup of a curve contained in $S$. 
	\end{prop}
	
	\begin{proof}
		By \cite[Proposition 3.1]{ACM}, the center of the divisorial contraction $f\colon X \to \PP^3$ is a curve $C\subset S$.
		By \cite[Proposition 1.2]{Tziolas}, $f\colon X \to \PP^3$ is the blowup of $\PP^3$ along $C$.
	\end{proof}

	\subsection{Sarkisov links centered on curves on quartic surfaces} 
	In this subsection, we give some constraints on curves $C\subset \PP^3$ whose blowups initiate Sarkisov links.
	While these curves are not completely classified in general, our main result is a classification of curves contained in smooth quartic surfaces with Picard rank $2$ whose blowups initiate Sarkisov links (Proposition~\ref{prop:blancLamyCurves}). 
	Recall from Lemma~\ref{lemma:rank2fibration} that the blowup $X \to \PP^3$ of a curve $C \subset \PP^3$ initiates a Sarkisov link if and only if $X \to \Spec(\CC)$ is a rank $2$ fibration.
	We begin with a degree bound:

	\begin{lem}\label{lem:degreeBound}
		Let $C \subset \PP^3$ be a curve, and $X \to \PP^3$ the blowup of $\PP^3$ along $C$.
		If $X \to \Spec(\CC)$ is a rank $2$ fibration, then $\deg(C) < 16$.
	\end{lem}

	\begin{proof}
		We denote by $E$ the exceptional divisor of the blowup $X \to \PP^3$, and by $H$ the pullback of the hyperplane class of $\PP^3$. 
		By definition of a rank $2$ fibration, $-K_X=4H-E$ is big.
		We take $n$ sufficiently large so that $|-nK_X|$ has no base components, and $-nK_X\sim A + F$, with $A$ very ample and $F$ effective.
		For any $T\in |-nK_X|$, we denote by $\check{T}$ its pushforward to $\PP^3$. Note that $\check{T}$ is a surface of degree $4n$ containing $C$ with multiplicity at least $n$. 
		Let $T_1, T_2\in |-nK_X|$ be general members.
		We claim that $T_1\cap T_2\not\subset E$. Indeed, if $T_1\in |-nK_X|$ and $D\in |A|$ are general members, then $E\cap T_1\cap D$ consists of finitely many points, and so $T_1\cap D\not\subset E$. If we take $T'_2=D+F\in |-nK_X|$, then $T_1\cap T'_2\not\subset E$. This proves the claim. 
		Therefore, as a $1$-cycle,
		\[
		{\check{T}}_1 \cdot {\check{T}}_2 = n^2 C+C',
		\]   
		with $C'$ a nonzero effective cycle. By B\'ezout's Theorem, $\deg\left({\check{T}}_1 \cdot {\check{T}}_2\right)=16n^2$, and thus
		\[
		16n^2 = \deg\left({\check{T}}_1 \cdot {\check{T}}_2\right) =  \deg(C)n^2 + \deg(C') > \deg(C)n^2.
		\]
		Hence, $\deg(C) < 16$. 
	\end{proof}

	The remaining of this section is devoted to classifying  curves that are contained in smooth quartic surfaces with Picard rank $2$ and initiate Sarkisov links.
	We first exclude curves that are complete intersections.
	
	\begin{lem}\label{lem:noCI}
		Let $S\subset \PP^3$ be a quartic surface, $C=S\cap T$ the complete intersection of $S$ with another surface $T\subset \PP^3$, and $X \to \PP^3$ the blowup of $\PP^3$ along $C$. Then $X \to \Spec(\CC)$ is not a rank $2$ fibration.
	\end{lem}

	\begin{proof}
		Any curve $\Gamma \subset T$ not contained in $S$ intersects $S$ in $4\deg(\Gamma)$ points, counted with multiplicities.
		They are therefore $4\deg(\Gamma)$-secant to $C$ and we conclude by Remark \ref{rem:quadrisecants}.
	\end{proof}

	When we consider an arbitrary curve $C$ contained in a surface $S\subset \PP^3$, it is useful to be able to replace $C$ with a better curve $C'\subset S$ that is linearly equivalent to $C$ in $S$. For instance, we may want to take $C'$ smooth. The next results allow us to compare the blowup of $\PP^3$ along $C$ with the blowup of $\PP^3$ along $C'$.

	\begin{lem}\label{lem:rank2FibrIff}
		Let $S$ be a smooth quartic surface,  and $C, C'\subset S$ (irreducible and reduced) curves that are linearly equivalent in $S$. 
		Denote by $X\to \PP^3$ and $X'\to \PP^3$ the blowups of $\PP^3$ along $C$ and $C'$, and by $\tilde{S}$ and $\tilde{S'}$ the strict transforms of $S$ on $X$ and $X'$, respectively.
		For any curve 
		$\check{\Gamma} \subset S$, denote by $\Gamma$ and $\Gamma'$ its image in $\tilde{S}$ and $\tilde{S'}$ under the identifications $S \isom \tilde{S}$ and  $S \isom \tilde{S'}$, respectively.
		Assume that $X$ and $X'$ are $\QQ$-Gorenstein.
		Then the following hold: 
		
		\begin{enumerate}
			\item\label{item:intersectionsAreSame} $(-K_{X})\cdot \Gamma = (-K_{X'}) \cdot \Gamma'$.
			\item\label{item:weakFanoIff} $X$ is weak Fano if and only if so is $X'$.
			\item\label{item:rank2FibrIff} 
			
			Suppose that $X$ is weak Fano and $X \to \Spec(\CC)$ is a rank $2$ fibration.
			If $C'$ is very general in $|C|$, then $X'$ is weak Fano and $X' \to \Spec(\CC)$ is a rank $2$ fibration.
		\end{enumerate}
	\end{lem}

	\begin{proof}
		We denote by $E$ and $E'$ the exceptional divisors of $X\to \PP^3$ and $X'\to \PP^3$, respectively.
		We abuse notation and use the same symbol $H$ to denote the hyperplane class in $\PP^3$ and its pullbacks to $X$ and $X'$.
		We shall see that, for any $a,b\in \ZZ$, $(aH- bE)\cdot \Gamma = (aH- bE') \cdot \Gamma'$.
		This yields \eqref{item:intersectionsAreSame} as a special case. 
		\begin{align*}
			(aH- bE)\cdot \Gamma =  (aH- bE)|_{\tilde S} \cdot \Gamma &= (aH|_S - bC)\cdot \check \Gamma \\
			&=(aH|_S - bC')\cdot \check \Gamma = (aH- bE')|_{\tilde S'} \cdot \Gamma'  = (aH- bE')\cdot \Gamma'.
		\end{align*}

		If $\Gamma \subset X$ is any curve with $(-K_X)\cdot \Gamma <0$, then $\Gamma \subset \tilde{S}$.
		By \eqref{item:intersectionsAreSame} $(-K_{X'})\cdot \Gamma' <0$.
		By the symmetric nature of the argument, we get that $-K_X$ is nef if and only if so is $-K_{X'}$. 
		Similarly, for bigness, we have
		\[
		(-K_X)^3 
		= (-K_X|_{\tilde S})^2 
		= (4H|_{\tilde{S}} - C)^2
		= (4H|_{\tilde{S'}} - C')^2 = (-K_{X'}|_{\tilde{S'}})^2 = (-K_{X'})^3
		\]
		In particular $(-K_X)^3>0$ if and only if $(-K_{X'})^3>0$. 
		This gives \eqref{item:weakFanoIff}.

		For \eqref{item:rank2FibrIff}, assume that $X$ is weak Fano and $X \to \Spec(\CC)$ is a rank $2$ fibration.
		Choosing $C'\in |C|$ smooth and using \eqref{item:weakFanoIff} and Lemma~\ref{lemma:weak-fano&rank2fibration},
		it suffices to prove that, for general $C'$, the morphism to the anti-canonical model of $X'$ is small.
		By Remark~\ref{rem:quadrisecants}, it is enough to prove that the subset of $|C|$ consisting of curves $C'$ for which there are 
		infinitely many curves of degree $\delta$ that are $4\delta$-secants to $C'$ is a union of countably many proper closed subsets.
		
		Denote by $\mathcal{H}_{g,\delta}$ the Hilbert scheme of curves of arithmetic genus $g$ and degree $\delta$.
		Consider the following closed subscheme of $|C|\times \mathcal{H}_{g,\delta}$:
		\[
		\mathcal{Q}_{g,\delta} \defeq \left\{ 
		(D,\gamma) \in |C| \times \mathcal{H}_{g,\delta} \, | \, \gamma \text{ is at least $4\delta$-secant to $D$} 
		\right\},
		\]
		together with the projection $\pi_{g,\delta}$ to the first factor.
		By Lemma~\ref{lemma:weak-fano&rank2fibration} and Remark~\ref{rem:quadrisecants}, there are finitely many pairs $(g_1,\delta_1),\dots, (g_k,\delta_k)$ such that
		\[
		\pi_{g,\delta}^{-1}(C) = \left\{
		\begin{array}{ll}
			\text{finite,} & \text{ if $(g,\delta) \in \big\{(g_1,\delta_1),\dots,(g_k,\delta_k)\big\},$} \\
			\emptyset, &  \text{ otherwise.}
		\end{array}
		\right.
		\]
		Thus, if $(g,\delta) \not\in \big\{(g_1,\delta_1),\dots,(g_k,\delta_k)\big\}$, then $\pi_{g,\delta}$ is not dominant; denote by $R_{g,\delta}\subset |C|$ its image.
		For $(g,\delta) \in \big\{(g_1,\delta_1),\dots,(g_k,\delta_k)\big\}$, by the upper-semicontinuity of the dimension of fibers of $\pi_{g,\delta}$, the subset $R_{g,\delta}$ of $|C|$ consisting of elements with positive dimensional fibers is proper and closed.
		Let $R$ be the union of all the subsets $R_{g,\delta}$.
		Then any $C' \in |C|\setminus R$ admits only finitely many  $4\delta$-secant curves of degree $\delta$.
	\end{proof}

	We end this section with a classification of curves contained in smooth quartic surfaces with Picard rank $2$ whose blowups initiate Sarkisov links.
	
	\begin{prop}\label{prop:blancLamyCurves}
		Let $C\subset \PP^3$ be a (possibly singular) curve of arithmetic genus $p_a$ and degree $d$ lying on a smooth quartic surface $S$ with Picard rank $2$.
		Let $X$ be the blowup of $\PP^3$ along $C$, and suppose that $X \to \Spec(\CC)$ is a rank $2$ fibration. Then $X$ is weak Fano and $(p_a,d)$ is one of the pairs in the list \eqref{list} of Theorem~\ref{thm:blancLamyCurves}.
	\end{prop}
	
	\begin{proof}
		We denote by $S$ a smooth quartic surface with Picard rank $2$ containing $C$, and by $\tilde{S}$ its strict transform  on $X$. 
		Suppose that $X \to \Spec(\CC)$ is a rank $2$ fibration. We will show that $X$ is weak Fano.
		Suppose that $X$ is not weak Fano. 
		Then the Sarkisov link initiated by $X \to \PP^3$ proceeds with an anti-flip. 
		By Lemma~\ref{lem:antiflipsRationalCurves}\eqref{item:generationBySmoothRationalCurves}, the extremal ray corresponding to the associated small contraction of $X$ is generated by a smooth rational curve $\Gamma\subset X$ such that $\tilde{S}\cdot \Gamma = -K_X \cdot \Gamma < 0$.
		In particular, $\Gamma \subset \tilde{S}$.
		By \cite[Lemma 4.4]{ACM}, we must have $-K_X \cdot \Gamma = -1$.
		Denote by $\check{\Gamma}$ the image of $\Gamma$ in $\PP^3$.
		Since $X \to \Spec(\CC)$ is a rank $2$ fibration,  Lemma \ref{lem:noCI} implies that $C$ is not a complete intersection.
		As in Section \ref{sec:S_with_rho=2}, we assume that $\Pic(S) = \ZZ H \oplus \ZZ W$, where $H$ is a hyperplane section of $S$,  the intersection product is given by the matrix (\ref{matrixbilinearform}) and $S$ has discriminant $r=b^2-8c>0$. By changing $W$ if necessary, we may assume that $0<b<16$. 
		Write $\check{\Gamma} = \alpha H + \beta W$ and $C=\delta H+\gamma W$  in $\Pic(S)$ with $\alpha, \beta, \delta, \gamma\in \ZZ$. 
		The conditions that $\check{\Gamma}$ is a rational curve and $-K_{X}\cdot \Gamma = -1$ give the following system:
		\begin{equation}\label{eq:system}\arraycolsep=2pt
			\left\{ 
			\begin{array}{rl}
				0 &<H\cdot\check{\Gamma}=4\alpha+b\beta\\
				-2 &= \check{\Gamma}^2 = 4\alpha^2 + 2b\alpha \beta + 2c\beta^2\\
				-1 &= (4H-C)\cdot \check{\Gamma} = \left(16-d\right)\alpha + \left(4b - \delta b+\frac{d^2-8(p_a-1)-\gamma^2 b^2}{4\gamma}\right)\beta.
			\end{array}
			\right.
		\end{equation}  
		By Lemma \ref{lem:degreeBound} and Theorem~\ref{thm:gdBounds}, $d < 16$ and $p_a \leq \frac{d^2}{8}$. Moreover, from $d^2-8(p_a-1)=r\gamma^2$ and the fact that  $d<16$ and $p_a\geq 0$, it follows that $r,\gamma^2\leq233$. Therefore, there are finitely many possibilities for the pair $(p_a,d)$, and hence for the integers $b$, $c$, $\delta$ and $\gamma$. 
		One may verify that, subject to these conditions,
		the system \eqref{eq:system} admits an integer solution $(\alpha, \beta)$ if and only if $(p_a,d)=(15,11)$. In this case, the discriminant of $S$ is $r=9$ and $\gamma=\pm 1$. So $\Pic(S)=\ZZ H\oplus \ZZ C$, and the integer solution of the system \eqref{eq:system} 
		gives a line $\ell=\check \Gamma \sim 3H-C$ contained in $S$. 
		We will show that the rational contraction $X\dasharrow Y$ to the anti-canonical model of $X$ contracts a divisor. 
		
		By Corollary \ref{cor:linearSystem}, a general element $D \in |C|$ is smooth and of genus and degree $(15,11)$.
		We may then compute that
		\begin{align*}
			&h^0\big(\PP^3,\II_D(3)\big) \geq h^0\big(\PP^3,\OO_{\PP^3}(3)\big) - h^0\big(D,\OO_D(3H)\big) = 20 - 19 =1;\\
			&h^0\big(\PP^3,\II_D(5)\big) \geq h^0\big(\PP^3,\OO_{\PP^3}(5)\big) - h^0\big(D,\OO_D(5H)\big) = 56 - 41 = 15.
		\end{align*}
		Here $h^0\big(D,\OO_D(3H)\big)$ and $h^0\big(D,\OO_D(5H)\big)$ are computed using Riemann-Roch. 
		Since the numbers $h^0(\PP^3,\II_C(n))$ can only increase under specialization, we conclude that:
		\[
		h^0\big(\PP^3,\II_C(3)\big) \geq 1 \, \text{ and } h^0\big(\PP^3,\II_C(5)\big) \geq 15.
		\]
		For degree reasons, the cubic $T$ containing $C$ must be unique, and $S \cap T = C\cup \ell$.
		Since $C\cup \ell$ is a complete intersection of type $(3,4)$ we may compute that
		\[
		h^0\big(\PP^3,\II_{C\cup \ell}(5)\big) = \binom{3+2}{3} + \binom{3+1}{3} = 14.
		\]
		Therefore we conclude that there exists a quintic surface $Q$ containing $C$ but not $\ell$.
		Moreover, since $C$ and $\ell$ intersect in $5$ points, all points of intersection of $\ell$ with $Q$ lie on $C$.
		This implies that the ideal of $C$ in $\PP^3$ is $\mathcal{I}_C = (f_3,f_4,f_5)$, where $f_i$ are homogeneous polynomials of degree $i$, and $f_3$, $f_4$ and $f_5$ cut out $T$, $S$ and $Q$, respectively.          
		The anti-canonical rational contraction of $X$ factors through the rational map of $\PP^3$ given by the global sections of $\mathcal{I}_{C}(4)$. We may choose coordinates and a basis of $H^0(\PP^3,\mathcal{I}_{C}(4))$ so that this map $\PP^3 \rmap  \PP^4$ is given by $\left(x_0f_3:x_1f_3:x_2f_3:x_3f_3:f_4\right)$, and its image is the singular hypersurface $Y\subset \PP^4$	cut out by the equation
		\[
		f_4(y_0, y_1, y_2, y_3) - y_4f_3(y_0, y_1, y_2, y_3) = 0.
		\]
		This map is divisorial, contracting the strict transform $\widetilde{T}$ of $T$ in $X$ to the singular point $(0:0:0:0:1)\in Y$.
		
		We have shown that the rational contraction $X\dasharrow Y$ to the anti-canonical model of $X$ contracts the divisor $\widetilde{T}$. 
		It follows that the $\widetilde{T}$-MMP for $X$ will end with an anti-canonical contraction $X'\to Y'$ that contracts a divisor, namely the strict transform of $\widetilde{T}$ in $X'$. The image $Y'$ of such a contraction has worse than terminal singularities. This contradicts the assumption that $X \to \Spec(\CC)$ is a rank $2$ fibration, as it violates condition \eqref{item:RRF3} in Definition \ref{def:rankRFibrations}. 
		From this contradiction we conclude that $X$ is weak Fano.
		
		Finally, by Lemma~\ref{lem:rank2FibrIff}, for a general $C' \in |C|$ the blowup $X' \to \PP^3$ of $\PP^3$ along $C'$ is weak Fano and $X' \to \Spec(\CC)$ is a rank $2$ fibration.  
		Furthermore, by Corollary~\ref{cor:linearSystem}, we may choose $C'$ to be smooth. Theorem~\ref{thm:blancLamyCurves} then implies that $\left(g(C'),\deg(C')\right)=(p_a,d)$ appears in \eqref{list}.
	\end{proof}

	\section{Gizatullin's problem for quartics with Picard rank $2$} \label{section:main}

	Let $S\subset \PP^3$ be a smooth quartic surface with Picard rank $2$, and $\varphi\colon \PP^3 \dasharrow \PP^3$ a birational map.
	Recall from Remark~\ref{rem:decomposition_group} that  
	$\varphi(S) = S$ if and only if $\varphi$ is volume preserving with respect to the Calabi-Yau pair $(\PP^3,S)$.
	Therefore, in order to investigate Problem~\ref{G-problem}, we may use the volume preserving Sarkisov program introduced in the previous section.

	\begin{prop}\label{reduction}
		Let $S \subset \PP^3$ be a smooth quartic surface with Picard rank $2$.
		Suppose that $\Bir(\PP^3;S)\neq \Aut(\PP^3;S)$. 
		Then there is a smooth curve $C \subset S$ of genus $g$ and degree $d$ such that the pair $(g,d)$ belongs to the list \eqref{list} of Theorem~\ref{thm:blancLamyCurves}.
	\end{prop}
	
	\begin{proof}
		Suppose that there exists a birational map $\varphi\in \Bir(\PP^3;S)\setminus \Aut(\PP^3;S)$.
		By Theorem \ref{thm:vp_SarkisovProgram}, there exists a factorization of $\varphi$ as a composition of volume preserving Sarkisov links.
		Since $\rho(\PP^3)=1$, the first Sarkisov link in the decomposition necessarily starts with a volume preserving divisorial contraction $X \to \PP^3$. By Proposition~\ref{CinS}, $X \to \PP^3$ is the blowup of $\PP^3$ along a curve $C' \subset S$.
		By Lemma~\ref{lemma:rank2fibration}, $X \to \Spec(\CC)$ is a rank $2$ fibration, and so
		Proposition~\ref{prop:blancLamyCurves} implies that  $(p_a(C'),\deg(C'))$ belongs to the list \eqref{list}. 
		A general member $C$ in the linear system $|C'|$ of $S$ is a smooth curve of genus $g$ and degree $d$ with  
		$(g,d)=(p_a(C'),\deg(C'))$, and the result follows. 
	\end{proof}
	
	In what follows, we adopt the notation introduced in Section~\ref{sec:S_with_rho=2}. In particular, we denote by $H$ the class of a hyperplane section of $S\subset \PP^3$.
	Given a curve $C\subset S$,  the arithmetic genus $p_a$ and degree $d$ of $C$ are given by 
	$(p_a,d)=\left(\frac{C^2}{2}+1,C\cdot H\right)$.
	
	\begin{cor}\label{cor:noRealizationForr>57}
		Let $S$ be a smooth quartic surface with Picard rank $2$ and discriminant $r$. If $r > 57$ or $r =52$, then
		\[
		\Bir(\PP^3;S)=\Aut(\PP^3;S).
		\]
	\end{cor}

	\begin{proof}  
		Let $S$ be a smooth quartic surface with Picard rank $2$ and discriminant $r$.
		Recall from \eqref{matrixbilinearform} in Section~\ref{sec:S_with_rho=2} that $r \equiv 0,1,4 \ (\modd{8})$. Suppose that there exists a birational map $\varphi\in \Bir(\PP^3;S)\setminus \Aut(\PP^3;S)$.
		By Proposition~\ref{reduction}, there is a curve $C\subset S$ with arithmetic genus $p_a$ and degree $d$ satisfying $(p_a,d)=\left(\frac{C^2}{2}+1,C\cdot H\right)\in\eqref{list}$. 
		Consider the sublattice $L$ of $\Pic(S)$ spanned by $H$ and $C$. 
		It has rank $2$ and discriminant $r'=(C\cdot H)^2-4C^2 = d^2 - 8(p_a-1)$.
		We compute $r'$ for each pair $(p_a,d)\in\eqref{list}$, and list the possible values of $r'$ in the same order as the corresponding pair in \eqref{list}:
		\[
		r' \in 
		\left\{
		\def\arraystretch{1.2}
		\begin{array}{c}
			\arraycolsep=1.4pt
			\def\arraystretch{1.2}
			\begin{array}{cccccccccccc}
				9, & 12, & 17, & 24, & 33, & 44, & 57, & 9, & 16, & 25, & 36, & 49,\\
				17, & 28, & 41, & 56, & 20, & 33, & 48, & 12, & 25, & 40, & 17, & 32, 
			\end{array}\\
			\arraycolsep=3pt
			\begin{array}{cccccccccc}
				24, & 41, & 16, & 33, & 25, & 17, & 9, & 28, & 20, &  17
			\end{array}
		\end{array}
		\right\}.
		\]
		
		Since $r=\disc\big(\Pic(S)\big)$ must divide $r'=\disc(L)\leq 57$, we conclude that $r \leq 57$.
		Among the integers $r$ with $r \equiv 0,1,4 \ (\modd{8})$ and $r \leq 57$, the only one that does not divide any $r'$ in the above list 
		is $r = 52$.
	\end{proof}

	For quartic surfaces with Picard rank $2$, Corollary~\ref{cor:noRealizationForr>57} reduces
	Problem~\ref{G-problem} to  surfaces $S$  with discriminant $r \leq 57$ and $r\neq 52$. 
	Recall from Proposition~\ref{prop:Autforrleq57} the possible values of $r$ in this case:
	\begin{itemize}
		\item[] $\mathcal{R}_0 = \{9, 12, 16, 24, 25, 33, 36, 44, 49, 57\}$,
		\item[] $\mathcal{R}_1 = \{17, 41\}$,
		\item[] $\mathcal{R}_2 = \{28, 56\}$, and 
		\item[] $\mathcal{R}_3 = \{20, 32, 40, 48\}$. 
	\end{itemize}
	Moreover, in each of these cases we have determined the finite index subgroup $\Aut^{\pm}(S)$ of $\Aut(S)$:
	\[
	\Aut^{\pm}(S) \cong 
	\left\{
	\begin{array}{ll}
		\{id\},            & \text{ if } r \in \mathcal{R}_0,\\
		\ZZ_2,            & \text{ if } r \in \mathcal{R}_1,\\
		\ZZ_2 \ast \ZZ_2, & \text{ if } r \in \mathcal{R}_2, \\
		\ZZ,              & \text{ if } r \in \mathcal{R}_3.    
	\end{array}   
	\right.
	\]
	
	Next we show that whenever $r \in \mathcal{R}_1 \cup \mathcal{R}_2 \cup \mathcal{R}_3$ we can find a curve $C\subset S$ as in Proposition~\ref{reduction} such that $\Pic(S) = \ZZ H \oplus \ZZ C$, where $H$ denotes the class of a hyperplane section. 
	This allows us to describe explicitly the generators of $\Aut^{\pm}(S)$ via their action on $\Pic(S)$ in each case.

	\begin{prop}\label{prop:boldcurvesasgenerator}
		Let $S$ be a smooth quartic surface with $\rho(S) = 2$ and discriminant $r \in \mathcal{R}_1 \cup \mathcal{R}_2 \cup \mathcal{R}_3$.
		\begin{enumerate}
			\item\label{item1} There is a smooth curve $C \subset S$ of genus $g$ and degree $d$ such that $\Pic(S) = \ZZ H \oplus \ZZ C$, where $(g,d)$ depends on $r$ as described in the following table:
			\[\label{table:boldCurves}\tag{C}
			\arraycolsep=7pt \def\arraystretch{1.2}
			\begin{array}{|c|c|c|c|c|c|c|c|c|}
				\rowcolor{gray!25}
				\hline
				r      & 17      & 41    & 28      & 56    & 20      & 32    & 40    & 48    \\
				\hline
				&&&&&&&&\\[-12pt]
				(g,d)  & (14,11) & (6,9) & (10,10) & (2,8) & (11,10) & (5,8) & (4,8) & (3,8) \\[1pt]
				\hline \rowcolor{gray!7}
				&\multicolumn{2}{c|}{}&\multicolumn{2}{c|}{}&\multicolumn{4}{c|}{}\\[-12pt] \rowcolor{gray!7}
				\Aut^{\pm}(S) & \multicolumn{2}{c|}{\ZZ_2} &\multicolumn{2}{c|}{\ZZ_2 \ast \ZZ_2} &\multicolumn{4}{c|}{\ZZ} \\[1pt]
				\hline
			\end{array}
			\]
			\item\label{item2} With respect to the basis $\{H , C\}$, the action of $\Aut^{\pm}(S)$ on $\Pic(S)$ is described as follows:
			\[\label{table:matricesOfGenerators}\tag{M}
			\arraycolsep=7pt \def\arraystretch{1.2}
			\begin{array}{|c|c|c|c|c|}
				\rowcolor{gray!25}
				\hline
				r & 17 & 28 & 20 & 40 \\
				\hline
				&&&&\\[-11pt]
				\Aut^{\pm}(S)  
				&
				\left\langle 
				\left(
				\begin{smallmatrix*}[r]
					19 & 72\\
					-5 & -19
				\end{smallmatrix*}
				\right)
				\right\rangle
				& 
				\left\langle 
				\left(
				\begin{smallmatrix*}[r]
					23 & 88\\
					-6 & -23
				\end{smallmatrix*}
				\right),
				\left(
				\begin{smallmatrix*}[r]
					-7 & -8\\
					6 & 7
				\end{smallmatrix*}
				\right)
				\right\rangle
				& 
				\left\langle 
				\left(
				\begin{smallmatrix*}[r]
					29 & 40\\
					-8 & -11
				\end{smallmatrix*}
				\right)
				\right\rangle
				& 
				\left\langle 
				\left(
				\begin{smallmatrix*}[r]
					43 & 18\\
					-12 & -5
				\end{smallmatrix*}
				\right)
				\right\rangle
				\\[2pt]
				\hline
				\rowcolor{gray!25}
				r & 41 & 56 & 32 & 48 \\
				\hline
				&&&&\\[-11pt]
				\Aut^{\pm}(S) 
				& 
				\left\langle 
				\left(
				\begin{smallmatrix*}[r]
					27 & 104\\
					-7 & -27
				\end{smallmatrix*}
				\right)
				\right\rangle
				& 
				\left\langle 
				\left(
				\begin{smallmatrix*}[r]
					31 & 120\\
					-8 & -31
				\end{smallmatrix*}
				\right),
				\left(
				\begin{smallmatrix*}[r]
					-1 & 0\\
					8 & 1
				\end{smallmatrix*}
				\right)
				\right\rangle
				& 
				\left\langle 
				\left(
				\begin{smallmatrix*}[r]
					41 & 24\\
					-12 & -7
				\end{smallmatrix*}
				\right)
				\right\rangle
				& 
				\left\langle 
				\left(
				\begin{smallmatrix*}[r]
					209 & 56\\
					-56 & -15
				\end{smallmatrix*}
				\right)
				\right\rangle \\[2pt]
				\hline
			\end{array}
			\]
		\end{enumerate}
	\end{prop}
	
	\begin{proof}
		To prove \eqref{item1}, write $\Pic(S) = \ZZ H \oplus \ZZ W$. As in \eqref{matrixbilinearform}, the intersection matrix of $S$ with respect to the basis $\{H,W\}$ can be written as
		\[
		Q = 
		\begin{pmatrix}
			4 & b\\
			b & 2c
		\end{pmatrix} \, ,
		\]
		so that $r = b^2 - 8c$.
		Notice that for each $r \in \mathcal{R}_1 \cup \mathcal{R}_2 \cup \mathcal{R}_3$, the corresponding pair $(g,d)$ listed on table \eqref{table:boldCurves} satisfies $r=d^2-8(g-1)$.
		Combining the two equalities we get
		\[
		(b+d)(b-d) = b^2 - d^2 = 8\big(b - (g-1)\big),
		\]
		from which we deduce that $4$ divides one of the two factors on the left.
		Then one may check that the divisor
		\[
		D = \alpha H + \beta W, \text{ with } (\alpha,\beta) = \left(\frac{b+d}{4},-1\right) \text{ or } \left(\frac{b-d}{4},1\right)
		\]
		satisfies $H\cdot D = d$ and $D^2 = 2g-2$.
		By Corollary \ref{cor:linearSystem}, there is a smooth curve $C \in |D|$.
		Since $\disc\big(\Pic(S)\big)=r=d^2-8(g-1)=\disc(\ZZ H \oplus \ZZ C)$, we conclude that $\Pic(S) = \ZZ H \oplus \ZZ C$.

		To prove \eqref{item2}, we first notice that each matrix in the table represents an isometry $\phi\in O(\Pic(S))$ since $\phi^{T}Q\phi=Q$.
		Moreover, in each case, either $(\phi-Id)Q^{-1}\in M_{2\times 2}(\ZZ)$ or $(\phi+Id)Q^{-1}\in M_{2\times 2}(\ZZ)$.
		Therefore, by Proposition~\ref{prop:Leeinfiniteordergen}, the isometry $\phi\in O(\Pic(S))$ is induced by an automorphism $g\in \Aut^{\pm}(S)$ if and only if $\phi(H)$ is ample. We now consider separately each case $r \in \mathcal{R}_i$ for $i\in \{1,2,3\}$.

		Suppose that $r\in \mathcal{R}_3$. 
		Then $\Aut^{\pm}(S)\cong \ZZ$ and $S$ does not contain rational curves (see Proposition~\ref{prop:possibilitiesForAut}). 
		We can check that $\phi (H)\cdot H>0$ and $\phi (H)^2>0$, and so $\phi (H)$ is ample by \cite[Chapter 8, Corollary 1.6]{Huybrechts}. 
		Therefore, the isometry $\phi\in O(\Pic(S))$ is induced by an automorphism $g\in \Aut^{\pm}(S)$. 
		For $r = 20, 32, 40, 48$, the corresponding minimal integer solutions  of \eqref{quadeq} are $(\alpha_1,\beta_1) = (4,5), (7,4), (43,18)$ and $(4,1)$, respectively.
		Moreover,  $\phi = h^k$, where $h$ is the matrix of Proposition~\ref{prop:Leeinvolutionsgen}\eqref{Leeinfinite} and $k = 3,2,1,4$, respectively.
		In each case, $\phi$ is the minimal power of $h$ satisfying the Gluing and Torelli conditions stated in Proposition~\ref{prop:Leeinvolutionsgen}, and so $g$ is the generator of $\Aut^{\pm}(S)$.

		Suppose that $r\in \mathcal{R}_2$. Then $\Aut^{\pm}(S)\cong \ZZ_2\ast \ZZ_2\cong\ZZ\rtimes\ZZ_2$ and $S$ does not contain rational curves (see Proposition~\ref{prop:possibilitiesForAut}). For a fixed $r\in \mathcal{R}_2$, we denote by $\phi_1$ and $\phi_2$ the two isometries  displayed in the table. 
		Exactly as in the previous case, we check that $\phi_i (H)$ is ample, $i\in\{1,2\}$, and so $\phi_1$ and $\phi_2$ are induced by automorphisms $g_1, g_2\in \Aut^{\pm}(S)$ respectively. 
		Notice that $\phi_1$ and $\phi_2$ have the form described in Proposition \ref{prop:Leeinvolutionsgen}\eqref{Leeinvolution}, and so $g_1$ and $g_2$ are involutions of $S$.
		To see that they are the generators of $\Aut^{\pm}(S)$, we check that $g_1g_2$ generates the maximal copy of $\ZZ$ in $\Aut^{\pm}(S)$.
		Indeed, for $r = 28, 56$,  the minimal solutions of \eqref{quadeq} are $(\alpha_1,\beta_1) = (23,27)$ and $(31,4)$, respectively.
		In both cases $\phi_1\phi_2 = h^2$, and $h^2$ is the minimal power of $h$ satisfying the Gluing and Torelli conditions stated in Proposition~\ref{prop:Leeinvolutionsgen}.

		Suppose that $r\in \mathcal{R}_1$. Then $\Aut^{\pm}(S)=\langle g\rangle \cong\ZZ_2$. 
		By Proposition~\ref{prop:chracterization_of_involutions}, $g^*$ is the reflection along the line generated by the unique ample class $A$ such that $A^2=2$.
		Taking $W=C$ in the proof of Proposition~\ref{prop:Autforrleq57}, our argument there shows that the divisor $A = 4H-C$ is ample and $A^2=2$.
		So $g^*:N^1(S)\to N^1(S)$ is the reflection given by:
		\[
		\alpha \mapsto (A\cdot \alpha)A - \alpha.
		\]
		One checks directly that, for each value of $r\in \mathcal{R}_1$, the isometry $\phi$ represented by the matrix in table \eqref{table:matricesOfGenerators} coincides with this reflection.
	\end{proof}

	\begin{rem}\label{rem:ample}
		In all cases of Proposition \ref{prop:boldcurvesasgenerator}, the divisor $D = 4H-C$ is ample.
		This appeared in the proof for $r \in \mathcal{R}_1$.
		For $r \in \mathcal{R}_2$ or $\mathcal{R}_3$, $D$ is nef and big, and
		$S$ contains no rational curves, so $D$ is automatically ample.
	\end{rem}

	Now that we have explicitly described the action of the generators of $\Aut^{\pm}(S)$ on $\Pic(S)$ when $r \in \mathcal{R}_1 \cup \mathcal{R}_2 \cup \mathcal{R}_3$, we proceed to construct Cremona transformations realizing them.
	We will use the following criterion to determine when $S \subset \PP^3$ is projectively equivalent to its image under a Cremona transformation of $\PP^3$.

	\begin{lem}\label{lem:projEquivalent}
		Let $\iota \colon S\hookrightarrow \PP^n$ be a subvariety embedded by a complete linear system $|H|$. 
		Let $\varphi \in \Bir(\PP^n)$ be a Cremona transformation whose restriction to $S$ is an isomorphism onto its image $S'=\varphi(S)\subset \PP^n$, and assume that $S'$ is embedded by a complete linear system in $\PP^n$.
		Then $S$ and $S'$ are projectively equivalent in $\PP^n$ if and only if there is an automorphism $g\in \Aut(S)$ fitting into a commutative diagram:
		\[
		\xymatrix@R=.5cm@C=1.3cm{
			\PP^n \ar@{-->}[r]^{\varphi} & \PP^n\\
			S \ar[u]^{\iota} \ar[ur] \ar[r]_{g} & \, S \, . \ar[u]_{\iota}
		}
		\]
	\end{lem}
	
	\begin{proof}
		Denote by $H'\in\Pic(S)$ the pullback of the hyperplane class of $\PP^n$ under the embedding $\varphi|_S\circ \iota$.
		By assumption,   $\varphi|_S\circ \iota\colon S\hookrightarrow \PP^n$ is given by the complete linear system $|H'|$.
		Hence, the condition that $g\in \Aut(S)$ fits into the commutative diagram above is equivalent to the condition $H'=g^*H$. 
	\end{proof}
	
	In order to construct Cremona transformations realizing the generators of $\Aut^{\pm}(S)$ described in Proposition~\ref{prop:boldcurvesasgenerator}, we will consider the Sarkisov links initiated by blowing up the curves $C\subset S$ listed in Proposition~\ref{prop:boldcurvesasgenerator}.
	In the case $r=20$, we will also need a curve $C' \in |4H-C|$, which has genus and degree $(g,d)=(3,6)$.
	We recall some numerics of these Sarkisov links, which can be recovered from \cite[Table 2 and Table 3]{Cutrone-Marshburn} and \cite[Example 4.7(ii)]{BL}.

	\begin{rem}[{\cite[Table 2 and Table 3]{Cutrone-Marshburn}, \cite[Example 4.7(ii)]{BL}}]\label{rem:links}
		Let $C\subset \PP^3$ be a smooth curve of genus $g$ and degree $d$, where $(g,d)$ is one of the pairs in \eqref{table:boldCurves} above or $(g,d)=(3,6)$.
		Suppose that $C$  satisfies conditions \eqref{genCond1} and \eqref{genCond2} of Theorem~\ref{thm:blancLamyCurves}.
		By Theorem~\ref{thm:blancLamyCurves}, the blowup $p\colon X\to \PP^3$ of $C$ initiates a Sarkisov link $\chi\colon \PP^3\dasharrow Y$ fitting into a diagram:
		\[
		\xymatrix@R=.25cm{
			X \ar@{..>}[rr]^{\phi} \ar[dd]_p  && X^+\ar[dd]^{p^+} \\
			\\
			\PP^3  \ar@{-->}[rr]^{\chi} && Y,
		}
		\]
		where $\phi$ is a flop or an isomorphism, $Y$ is a smooth Fano $3$-fold with $\rho(Y)=1$, and $p^+\colon X^+ \to Y$ is the blowup of $Y$ along a smooth curve of genus $g^+$ and degree $d^+$. Here, the degree is measured with respect to the ample generator of $\Pic(Y)$. 
		
		Denote by $H\in \Pic(X)$ the pullback of the hyperplane class of $\PP^3$, by $H^+\in \Pic(X^+)$ the pullback of the ample generator of $\Pic(Y)$, and by $E$ and $E^+$ the exceptional divisors of $p$ and $p^+$, respectively.
		With respect to the bases $\{H^+, E^+\}$ of $\Pic(X^+)$ and $\{H , E\}$ of $\Pic(X)$, the isomorphism $\phi^*$ takes the form
		\begin{equation*}
			\phi^*=
			\left(
			\begin{array}{cc}
				a & \frac{ac-1}{b} \\ 
				-b & -c
			\end{array}
			\right)
		\end{equation*}
		for suitable integers $a$, $b$ and $c$.
		
		For each pair $(g,d)$ in \eqref{table:boldCurves} above or $(g,d)=(3,6)$, the Fano $3$-fold $Y$, as well as the values of $g^+$, $d^+$, $a$, $b$ and $c$, are displayed in the following table:
		
		\begin{center}
			\resizebox{1 \textwidth}{!}{	
				$
				\arraycolsep=5pt \def\arraystretch{1.5}
				\begin{array}{|c|c|c|c|c|c|c|c|c|c|}
					\rowcolor{gray!25}
					\hline
					(g,d)  & (14,11) & (6,9) & (10,10) & (2,8) & (11,10) & (3,6) & (5,8) & (4,8) & (3,8) \\[1pt]
					\hline
					Y   & \PP^3 & \PP^3 & \PP^3 & \PP^3 & \PP^3 & \PP^3 & \PP^3 & X_5 & \PP^3 \\
					\rowcolor{gray!7}
					\hline 
					(g^+,d^+)  & (14,11) & (6,9) & (10,10) & (2,8) & (11,10) & (3,6) & (5,8) & (4,10) & (3,8) \\[1pt]
					\hline
					(a,b,c)  & (19,5,19) & (27,7,27) & (23,6,23) & (31,8,31) & (11,3,11) & (3,1,3) & (7,2,7) & (11,3,5) & (15,4,15) \\[1pt]
					\hline
				\end{array}
				$
			}
		\end{center}
	\end{rem}
	
	We are now ready to realize the generators of $\Aut^{\pm}(S)$ of a smooth quartic surface $S$ with Picard rank $2$ and discriminant $r \in \mathcal{R}_1 \cup \mathcal{R}_2 \cup \mathcal{R}_3$ as restrictions of Cremona transformations of $\PP^3$. 
	These Cremona transformations will be constructed as compositions of Sarkisov links initiated by blowing up smooth curves $C\subset S$ with invariants $(g,d)$ listed in Table \eqref{table:boldCurves}. 
	However, for the blowup of $C$ to initiate a Sarkisov link, we need that $C$ satisfies the generality conditions \eqref{genCond1} and \eqref{genCond2} of Theorem \ref{thm:blancLamyCurves}.

	\begin{lem}\label{lem:generalityIsAlwaysSatisfied}
		Let $S$ be a smooth quartic surface with Picard rank $2$ and discriminant $r \in \mathcal{R}_1 \cup \mathcal{R}_2 \cup \mathcal{R}_3$.
		Let $(g,d)$ be the pair of invariants in Table \eqref{table:boldCurves} corresponding to $r$, or $(g,d) = (3,6)$ if $r = 20$.
		Then there exists a smooth curve $C \subset S$ of genus and degree $(g,d)$, and any such curve
		satisfies conditions \eqref{genCond1} and \eqref{genCond2} of Theorem \ref{thm:blancLamyCurves}.
	\end{lem}
	
	\begin{proof}
		If $(g,d)$ is one of the pairs of Table \eqref{table:boldCurves}, then the existence of a smooth curve $C$ of genus and degree $(g,d)$ is guaranteed by Proposition \ref{prop:boldcurvesasgenerator}.
		If $r = 20$ and $(g,d) = (3,6)$, then $\Pic(S) = \ZZ H \oplus \ZZ W$, where $W$ is a curve of genus and degree $(11,10)$, again by Proposition \ref{prop:boldcurvesasgenerator}.
		By Corollary \ref{cor:linearSystem}, a general element $C \in |4H-W|$ is a smooth curve of genus and degree $(3,6)$.
		
		We now show that the $C$ satisfies the generality conditions \eqref{genCond1} and \eqref{genCond2} of Theorem \ref{thm:blancLamyCurves}.
		First note that, in all cases, the divisor $D = 4H -C$ is ample:
		for $(g,d)$ in Table \eqref{table:boldCurves} this is Remark \ref{rem:ample};
		for $(g,d) = (3,6)$, $D$ is a nef and big divisor and $S$ contains no rational curves (see the proof of Proposition \ref{prop:Autforrleq57}), therefore $D$ is ample.

		Suppose that $C$ does not satisfy condition \eqref{genCond1}, and let $\Gamma$ be a $5$-secant line, $9$-secant conic or $13$-secant twisted cubic.
		Let $X$ the blowup of $\PP^3$ along $C$, denote by $\tilde{S}$ the strict transform of $S$ in $X$, and by $\tilde\Gamma$ the strict transform of $\Gamma$ in $X$. 
		Then $\tilde{S} \cdot \tilde \Gamma = -1$, and thus $\tilde\Gamma \subset \tilde{S}$.
		Moreover, $\tilde S|_{\tilde S}\sim 4H -C = D$, and so $\tilde{S} \cdot \tilde \Gamma = (D \cdot \Gamma)_S = -1$, contradicting the ampleness of $D$.
		Therefore, $C$ satisfies condition \eqref{genCond1}, and thus the blowup $X$ of $\PP^3$ along $C$ is weak Fano.

		Now suppose that $C$ does not satisfy condition \eqref{genCond2}.
		This is equivalent to the morphism $X \to Z$ to the anti-canonical model of $X$ being divisorial.
		Denote by $F$ the exceptional divisor of this morphism, and write $F = \delta H - \mu E \in N^1(X)$, where we also denote by $H$ the pullback of the hyperplane class to $X$, and by $E$ the exceptional divisor of the blowup $X\to \PP^3$.
		Then     
		\begin{equation}\label{eq:restrIsTrivial}
			0 = (-K_X)^2\cdot F = D \cdot F|_{\tilde{S}} \implies 0 = F|_{\tilde{S}} = \delta H - \mu C,
		\end{equation}
		where the implication follows from the fact that $D$ is ample and $F$ is effective. 
		However,  $\Pic(S) = \ZZ H \oplus \ZZ C$, and so \eqref{eq:restrIsTrivial} can only be satisfied when $\delta = \mu =0$, i.e.,  when $F = 0$, a contradiction.
	\end{proof}

	\begin{prop}[$\ZZ_2$-case]\label{prop:realizationZ2}
		Let $S$ be a smooth quartic surface with $\rho(S) = 2$ and discriminant $r \in\{17, 41\}$.
		Then the restriction homomorphism $\Bir(\PP^3;S) \to \Aut(S)\isom \ZZ_2$ is surjective.
	\end{prop}

	\begin{proof}
		We first treat the case $r=17$. 
		By Proposition \ref{prop:boldcurvesasgenerator}, $\Pic(S)=\ZZ H\oplus\ZZ C$, where $C$ is a smooth curve of genus and degree $(14,11)$.   
		Denote by $X\rightarrow\PP^3$ the blowup of $\PP^3$ along $C$, and by $\tilde S\subset X$ the strict transform of $S$. 
		By Lemma \ref{lem:generalityIsAlwaysSatisfied} and Remark \ref{rem:links}, $X \to \PP^3$ initiates a Sarkisov link $\chi\colon \PP^3 \rmap \PP^3$.
		Using the notation of Remark \ref{rem:links}, we first prove that $\chi$ restricts to an isomorphism on $S$.
		Indeed, the restriction of $p$ to $\tilde S$ is clearly an isomorphism onto $S$.
		Moreover, for any curve $\gamma\subset X$ flopped by $\phi$, we have $\tilde S\cdot \gamma = -K_X\cdot \gamma = 0$.
		Since $-K_X|_{\tilde S} = 4H-C$ is ample on $\tilde S$ by Remark \ref{rem:ample}, $\gamma$ must be disjoint from $\tilde S$.
		Since $\phi$ preserves anti-canonical sections, the class of $S^+=\phi(\tilde S)$ on $X^+$ is $4H^+ - E^+$.
		So, for any curve $e^+\subset X^+$ contracted by $p^+$, $S^+ \cdot e^+ = 1$, and so $p^+$ restricts to an isomorphism on $S^+$.
		We thus conclude that $\chi|_S \colon S \rmap \chi(S)$ is an isomorphism.
		
		By Remark \ref{rem:links}, in terms of the bases $\{H^+,E^+\}$ for $N^1(X^+)$ and $\{H,E\}$ for $N^1(X)$,  the isomorphism $\phi^*\colon N^1(X^+) \to N^1(X)$ is given by the matrix 
		\begin{equation*}
			\phi^*=
			\left(
			\begin{array}{rr}
				19 & 72 \\ 
				-5 & -19
			\end{array}
			\right).
		\end{equation*}
		Notice that this is the same matrix as the one corresponding to the generator $\tau$ of $\Aut^\pm(S)$ in Table \eqref{table:matricesOfGenerators}, and $\Aut(S)=\Aut^\pm(S)$ by Proposition~\ref{prop:possibilitiesForAut}.
		In particular, $\tau^*H^+ = H$ and thus, by Lemma~\ref{lem:projEquivalent}, $\chi(S)$ is projectively equivalent to $S$.
		Up to composing it with an automorphism of $\PP^3$, we may assume that $\chi \in \Bir(\PP^3;S)$ and, by Corollary \ref{cor:actionDetermByPic}\eqref{it:actionDetermByPic2}, $\chi|_S = \tau$. 
		This proves that  the restriction homomorphism $\Bir(\PP^3;S) \to \Aut(S) = \langle \tau \rangle$ is surjective.
		
		For $r = 41$, we pick $C\subset S$ a smooth curve of genus and degree $(6,9)$, and follow the exact same argument.
		The numerics for the corresponding link are given again in Remark \ref{rem:links}.
	\end{proof}

	\begin{prop}[$\ZZ_2 \ast \ZZ_2$-case]\label{prop:realizationZ2*Z2}
		Let $S$ be a smooth quartic surface with $\rho(S) = 2$ and discriminant $r \in\{28, 56\}$.
		Then $\Aut^{\pm}(S)  \isom \ZZ_2 \ast \ZZ_2$ is contained in the image of the restriction homomorphism $\Bir(\PP^3;S) \to \Aut(S)$.
	\end{prop}
	
	\begin{proof}
		We first treat the case $r=28$. 
		By Proposition \ref{prop:boldcurvesasgenerator}, $\Pic(S)=\ZZ H\oplus\ZZ C_1$, where $C_1$ is a smooth curve of genus and degree $(10,10)$.
		Moreover,  
		$\Aut^{\pm}(S) \isom \ZZ_2 \ast \ZZ_2$ is generated by two involutions $\tau_1$ and $\tau_2$ that act on $\Pic(S)$ as
		\[
		\tau_1^* = 
		\left(
		\begin{array}{rr}
			23 &  88 \\
			-6 & -23
		\end{array}
		\right) \ \ \text{ and } \ \
		\tau_2^* =
		\left(
		\begin{array}{rr}
			-7 &  -8 \\
			6 & 7
		\end{array}
		\right)
		\]   
		with respect to the basis $\{H,C_1\}$.
		
		Denote by $X_1$ the blowup of $\PP^3$ along $C_1$. 
		By Lemma \ref{lem:generalityIsAlwaysSatisfied} and Remark \ref{rem:links}, $X_1 \to \PP^3$ initiates a Sarkisov link $\chi_1\colon \PP^3 \rmap \PP^3$.    
		Arguing exactly as in the proof of Proposition \ref{prop:realizationZ2}, we see that $\chi_1$ restricts to an isomorphism on $S$ and, after composing it with an automorphism of $\PP^3$, we may assume that $\chi_1(S)=S$ 
		and $\chi_1|_S = \tau_1$.
		
		As for the second generator $\tau_2$ of $\Aut^\pm(S)$, let $C_2$ be a smooth element of $|5H-C_1|$.
		Then $C_2$ is also of genus and degree $(10,10)$, and so the blowup $X_2\to \PP^3$ along $C_2$ again initiates a Sarkisov link $\chi_2 \colon \PP^3 \rmap \PP^3$. As before, after composing it with an automorphism of $\PP^3$, we may assume that $\chi_2(S) = S$, and the induced automorphism on $S$ acts on $\Pic(S)$ as
		\[
		\left(
		\begin{array}{rr}
			23 &  88 \\
			-6 & -23
		\end{array}
		\right)
		\]
		with respect to the basis $\{H,C_2\}$.
		Changing the basis to $\{H,C_1\}$, we get
		\[
		{\chi_2|}_S^* 
		=
		\left(
		\begin{array}{rr}
			1 &  5 \\
			0 & -1
		\end{array}
		\right)
		\left(
		\begin{array}{rr}
			23 &  88 \\
			-6 & -23
		\end{array}
		\right)
		\left(
		\begin{array}{rr}
			1 &  5 \\
			0 & -1
		\end{array}
		\right)^{-1}
		=
		\left(
		\begin{array}{rr}
			-7 &  -8 \\
			6 & 7
		\end{array}
		\right).
		\]
		So, by Corollary \ref{cor:actionDetermByPic}\eqref{it:actionDetermByPic2}, $\chi_i|_S = \tau_i$ for $i = 1,2$, and 
		thus $\Aut^{\pm} = \langle \tau_1, \tau_2 \rangle$ is contained in the image of the restriction homomorphism $\Bir(\PP^3;S) \to \Aut(S)$.

		The case $r = 56$ is analogous:
		we choose $C_1$ to be a smooth curve of genus and degree $(2,8)$ and $C_2$ a smooth element of $|4H - C_1|$.
		Then $C_2$ is also a curve of genus and degree $(2,8)$, and the construction above works verbatim.
	\end{proof}
	
	\begin{rem}
		The four links described in the proofs of Propositions~\ref{prop:realizationZ2} and \ref{prop:realizationZ2*Z2}, initiated by the blowup of $\PP^3$ along smooth curves of genus and degree $(14,11)$, $(6,9)$, $(10,10)$ and $(2,8)$, 
		were described in detail in \cite[Proposition 3.1 and Remark 3.2]{ZikRigid}. Each one is 
		a birational involution $\chi\colon \PP^3\dasharrow \PP^3$  fitting into a commutative diagram:
		\[
		\xymatrix@R=.3cm{
			X \ar[dd]_p \ar[rd] \ar@{..>}^{\phi}[rr] && X \ar[ld] \ar[dd]^p\\
			& Z \ar@(dl,dr)^{\alpha}\\
			\PP^3 \ar@{-->}_{\chi}[rr] && \ \PP^3 \ ,
		}
		\]
		where $\phi\colon X\psmap X$ is a flop, the anti-canonical model $Z$ of $X$ is a double cover of $\PP^3$ ramified along a sextic hypersurface, and $\alpha\colon Z\to Z$ is the deck transformation of $Z$ over $\PP^3$.
	\end{rem}

	\begin{prop}[$\ZZ$-case]\label{prop:realizationZ}
		Let $S$ be a smooth quartic surface with $\rho(S) = 2$ and discriminant $r \in\{20, 32, 40, 48\}$.
		Then $\Aut^{\pm}(S) \isom \ZZ$ is contained in the image of the restriction homomorphism $\Bir(\PP^3;S) \to \Aut(S)$.
	\end{prop}
	
	\begin{proof}
		We first treat the case $r = 20$.
		By Proposition \ref{prop:boldcurvesasgenerator}, there is a smooth curve $C_1\subset S$ of genus and degree $(11,10)$ such that $\Pic(S) = \ZZ H \oplus \ZZ C_1$.
		By Lemma \ref{lem:generalityIsAlwaysSatisfied} and Remark \ref{rem:links}, the blowup $X_1 \to \PP^3$ along $C_1$ initiates a Sarkisov link $\chi_1\colon \PP^3 \dasharrow \PP^3$. 
		Following the notation introduced in Remark \ref{rem:links}, we denote by $C_1^+\subset \PP^3$  the center of the blowup $p_1^+\colon X_1^+\to \PP^3$, and set $S_1\defeq \chi_1(S)\subset \PP^3$.
		Arguing as in the proof of Proposition \ref{prop:realizationZ2}, we see that ${\chi_1}|_S\colon S \to S_1$ is an isomorphism.    
		Set $\sigma_1\defeq{\chi_1}|_{S}\colon S \to S_1$.
		By Remark \ref{rem:links}, with respect to the bases $\{H^+,C_1^+\}$ of $\Pic(S_1)$ and $\{H,C_1\}$ of $\Pic(S)$, $\sigma_1^*$ takes the form
		\[
		\sigma_1^* =
		\left(
		\begin{array}{rr}
			11 &  40 \\
			-3 & -11
		\end{array}
		\right).
		\]
		Using Lemma~\ref{lem:projEquivalent}, one can check that $S$ and $S_1$ are not projectively equivalent.
		
		We will now perform a second link.    
		Let $C_2\subset S_1$ be a smooth element of the linear system $|4H^+ - C_1^+|$.
		Then $C_2$ is a curve of genus and degree $(3,6)$.
		Denote by $p_2\colon X_2 \to \PP^3$ the blowup of $\PP^3$ along $C_2$.
		By Lemma \ref{lem:generalityIsAlwaysSatisfied} and Remark \ref{rem:links}, it initiates a Sarkisov link $\chi_2 \colon \PP^3 \rmap \PP^3$.
		Once again, we follow the notation introduced in Remark \ref{rem:links}. We denote by $E_2^+\subset X_2^+$  the exceptional divisor of the blowup $p_2^+\colon X_2^+\to \PP^3$, by $C_2^+\defeq p_2^+(E_2^+)\subset \PP^3$ its center, and set $S_2\defeq \chi_2(S_1)\subset \PP^3$.
		Arguing as in the proof of Proposition \ref{prop:realizationZ2}, we see that the restriction $\chi_2|_{S_1}\colon S_1 \to S_2$ is an isomorphism.
		Set $\sigma_2\defeq{\chi_2}|_{S_1}\colon S_1 \to S_2$, and write $A$ and $A^+$ for the hyperplane classes of $S_1$ and $S_2$, respectively.
		By Remark \ref{rem:links}, with respect to the bases $\{A^+,C_2^+\}$ of $\Pic(S_2)$ and $\{A,C_2\}$ of $\Pic(S_1)$, $\sigma_2^*$ takes the form
		\[
		\sigma_2^* = 
		\left(
		\begin{array}{rr}
			3 &  8 \\
			-1 & -3
		\end{array}
		\right).
		\]
		A smooth element $C'\subset S_2$ in the linear system $|4A^+ - C_2^+|$ is a curve of genus and degree $(11,10)$.
		Computing the matrix of the composition $(\sigma_2\circ \sigma_1)^*\colon \Pic(S_2) \to \Pic(S)$ with respect to the bases $\{A^+,C'\}$ of $\Pic(S_2)$ and $\{H,C_1\}$ of $\Pic(S)$, we get:
		\[
		(\sigma_2\circ \sigma_1)^* = 
		\left(
		\begin{array}{rr}
			11 &  40 \\
			-3 & -11
		\end{array}
		\right)
		\left(
		\begin{array}{rr}
			1 &  4 \\
			0 & -1
		\end{array}
		\right)
		\left(
		\begin{array}{rr}
			3 &  8 \\
			-1 & -3
		\end{array}
		\right)
		\left(
		\begin{array}{rr}
			1 &  4 \\
			0 & -1
		\end{array}
		\right)^{-1}
		= 
		\left(
		\begin{array}{rr}
			29 &  40 \\
			-8 & -11
		\end{array}
		\right).
		\]
		Notice that this is the same matrix as the one corresponding to the generator of $\Aut^\pm(S)$ in Table \eqref{table:matricesOfGenerators}.
		By Lemma \ref{lem:projEquivalent},  after composing it with an automorphism of $\PP^3$, we may assume that $(\chi_2\circ \chi_1)(S) = S$ and, by Corollary \ref{cor:actionDetermByPic}\eqref{it:actionDetermByPic2}, the restriction $(\chi_2\circ \chi_1)|_S$ generates $\Aut^{\pm}(S)$.
		
		The cases $r = 32, 40, 48$ are analogous: 
		the birational map $\chi\colon \PP^3 \rmap \PP^3$ that stabilizes $S$ and generates $\Aut^{\pm}(S)\isom \ZZ$ is always the composition of two Sarkisov links $\chi_1$ and $\chi_2$. The first Sarkisov link $\chi_1\colon \PP^3 \rmap Y$ is initiated by the blowup of $\PP^3$ along a smooth curve $C_1\subset S$ of genus and degree $(g,d)$ indicated in the table below, while the second Sarkisov link $\chi_2\colon Y \rmap \PP^3$ is initiated by the blowup of $Y$ along a smooth smooth curve $C_2 \subset \chi_1(S)$ such that $C_2 \sim \alpha H^+ + \beta C_1^+$ for suitable integers $\alpha$ and $\beta$ listed in the table below. 
		\[
		\arraycolsep=7pt \def\arraystretch{1.5}
		\begin{array}{|c|c|c|c|c|c|}
			\hline \rowcolor{gray!25}
			r & C_1 &\chi_1 & (\alpha,\beta) & C_2 & \chi_2 \\
			\hline
			32 & (5,8) & \PP^3 \rmap \PP^3 & (4,-1) & (5,8) & \PP^3 \rmap \PP^3\\
			\hline \rowcolor{gray!7}
			40 & (4,8) & \PP^3 \rmap X_5 & (2,-1) & (4,10) & X_5 \rmap \PP^3\\
			\hline
			48 & (3,8) & \PP^3 \rmap \PP^3 & (4,-1) & (3,8) & \PP^3 \rmap \PP^3 \\
			\hline
		\end{array} 
		\]
	\end{proof}
	
	We are now ready to prove our main theorem:

	\begin{proof}[{Proof of Theorem~\ref{introthm}}]
		First suppose that $r > 57$ or $r = 52$.
		By Corollary \ref{cor:noRealizationForr>57}, $\Bir(\PP^3;S)=\Aut(\PP^3;S)$ and by Corollary~\ref{cor:actionDetermByPic}(3), $\Aut(\PP^3;S)=\{id\}$. 
		
		If $r \leq 57$, $r \neq 52$ and $\Aut(S) \neq \{id\}$, then $r \in \{17, 41\} \cup \{28, 56\} \cup \{20, 32, 40, 48\}$ by Proposition~\ref{prop:Autforrleq57}.
		For $r$ in each one of these three sets, we conclude by Propositions \ref{prop:realizationZ2}, \ref{prop:realizationZ2*Z2} and \ref{prop:realizationZ}, respectively.
	\end{proof}

\end{document}